\documentclass[a4paper,12pt]{amsart}

\usepackage{amsmath,amssymb,amsthm,stmaryrd}
\usepackage{hyperref} 
\usepackage[T1]{fontenc}
\usepackage{latexsym}
\usepackage{enumerate}
\usepackage{fancyhdr}
\usepackage{latexsym}
\usepackage{dsfont}				
\usepackage{mathrsfs}			
\usepackage[active]{srcltx}		
\usepackage{esint}
\usepackage{pdfsync}

\usepackage{mathrsfs}
\usepackage{graphicx}

\usepackage[usenames,dvipsnames]{xcolor}

\usepackage[nospace,noadjust]{cite}

\newtheorem{Def}{Definition}[section]
\newtheorem{Lemma}[Def]{Lemma}

\newtheorem{Cor}[Def]{Corollary}
\newtheorem{Theorem}[Def]{Theorem}
\newtheorem{Prop}[Def]{Proposition}
\newtheorem{Remark}[Def]{Remark}

\DeclareMathOperator{\supp}{supp}

\DeclareMathOperator{\loc}{loc}

\numberwithin{equation}{section}

\newcommand{\R}{\ensuremath{\mathbb{R}}}
\newcommand{\C}{\ensuremath{\mathbb{C}}}
\newcommand{\N}{\ensuremath{\mathbb{N}}}
\newcommand{\Z}{\ensuremath{\mathbb{Z}}}

\newcommand{\calA}{\ensuremath{\mathcal{A}}}

\newcommand{\calF}{\ensuremath{\mathcal{F}}}

\newcommand{\calS}{\ensuremath{\mathcal{S}}}

\def\Xint#1{\mathchoice
   {\XXint\displaystyle\textstyle{#1}}%
   {\XXint\textstyle\scriptstyle{#1}}%
   {\XXint\scriptstyle\scriptscriptstyle{#1}}%
   {\XXint\scriptscriptstyle\scriptscriptstyle{#1}}%
   \!\int}
\def\XXint#1#2#3{{\setbox0=\hbox{$#1{#2#3}{\int}$}
     \vcenter{\hbox{$#2#3$}}\kern-.5\wd0}}

\def\aver#1{\Xint-_{#1}}

\usepackage{comment}

\usepackage{color}

\definecolor{gr}{rgb}   {0.,   0.8,   0. } 
\definecolor{bl}{rgb}   {0.,   0.5,   1. } 
\definecolor{mg}{rgb}   {0.7,  0.,    0.7}

\title[]{$L^p$ estimates for wave equations with specific $C^{0,1}$ coefficients}
\author{Dorothee Frey}
\author{Pierre Portal}

\address{Dorothee Frey, Karlsruhe Institute of Technology, Department of Mathematics, 76128 Karlsruhe, Germany}
\email{dorothee.frey@kit.edu}

\address{Pierre Portal, Australian National University, Mathematical Sciences Institute, Hanna Neumann Building, Ngunnawal and Ngambri Country, 
Canberra ACT 2601, Australia}
\email{Pierre.Portal@anu.edu.au}

\begin{document}

\begin{abstract}

Peral/Miyachi's celebrated theorem on fixed time $L^{p}$ estimates with loss of derivatives for the wave equation states that the operator $(I-\Delta)^{- \frac{\alpha}{2}}\exp(i \sqrt{-\Delta})$ is bounded on $L^{p}(\R^{d})$ if and only if $\alpha \geq s_{p}:=(d-1)\left|\frac{1}{p}-\frac{1}{2}\right|$. We extend this result to operators  of the form  $\mathcal{L}  = -\sum \limits _{j=1} ^{d} a_{j+d}\partial_{j}a_{j}\partial_{j}$,   such that, for $j=1,...,d$, the functions $a_{j}$ and $a_{j+d}$ only depend on $x_{j}$, are bounded above and below, but are  merely Lipschitz continuous.  This is below the $C^{1,1}$ regularity that is known to be necessary  in general  for Strichartz estimates in dimension $d \geq 2$. Our proof is based on an approach to the boundedness of Fourier integral operators recently developed by Hassell, Rozendaal, and the second author. We construct a scale of adapted Hardy spaces on which $\exp(i\sqrt{ \mathcal{L}} )$ is bounded by lifting $L^{p}$ functions to the tent space $T^{p,2}(\R^{d})$, using a wave packet transform adapted to the Lipschitz metric induced by the coefficients $a_j$.  The result then follows from Sobolev embedding properties of these spaces.\\

\end{abstract}

\thanks{The research of D. Frey is partly supported by the Deutsche Forschungsgemeinschaft (DFG, German Research Foundation) -- Project-ID 258734477 -- SFB 1173. The research of P. Portal is partly supported by the Discovery Project DP160100941 of the Australian Research Council.}

\date{\today}

\maketitle


\section{Introduction}

In 1980, Peral \cite{Peral} and Miyachi \cite{Miyachi}  proved that the operator 
$(I-\Delta)^{-\frac{\alpha}{2}}\exp(i \sqrt{-\Delta})$ is bounded on $L^{p}(\R^{d})$ if and only if $\alpha \geq s_{p}:=(d-1)\left|\frac{1}{p}-\frac{1}{2}\right|$. Their result was then extended to general Fourier integral operators  (FIOs)  in a celebrated theorem of Seeger, Sogge, and Stein \cite{sss}, leading, in particular, to $L^{p}(\R^{d})$ well-posedness results for wave equations with smooth variable coefficients on $\R^{d}$ or driven by the Laplace-Beltrami operator on a compact manifold. To establish well-posedness of wave equations in more complex geometric settings, many results have been obtained in the past 30 years, using extensions of Peral/Miyachi's fixed time estimates with loss of derivatives, Strichartz estimates, and/or local smoothing properties. This includes Smith's  parametrix construction  \cite{parametrix},Tataru's Strichartz estimates \cite{tataru}   for wave equations on $\R^{d}$ with $C^{1,1}$ coefficients, and M\"uller-Seeger's extension of Peral-Miyachi's result to the sublaplacian on Heisenberg type groups \cite{ms}, as well as many other important results for specific operators, such as Laplace-Beltrami operators on symmetric spaces. \\

In this paper, we consider operators  of the form   $\mathcal{L}  = -\sum \limits _{j=1} ^{d} a_{j+d}\partial_{j}a_{j}\partial_{j}$,  such that, for $j=1,...,d$, the functions $a_{j}$ and $a_{j+d}$ only depend on $x_{j}$, are bounded above and below, and are Lipschitz continuous.   For these operators, we extend Peral/Miyachi's result by proving that $(I+  \mathcal{L})^{-\frac{\alpha}{2}}\exp(i \sqrt{\mathcal{L} })$ is bounded on $L^{p}(\R^{d})$ for $\alpha \geq s_{p}:=(d-1)\left|\frac{1}{p}-\frac{1}{2}\right|$. When $s_{p} \leq 2$, we show well-posedness for data in $W^{s_{p},p}(\R^{d})$, 
even when  $\mathcal{L}$  is perturbed by first order drift terms depending on all the variables (see Theorem \ref{thm:main1} and Section \ref{sec:perturb}). While the algebraic structure of the coefficient matrix is a serious limitation, the roughness of the coefficients is a satisfying and somewhat surprising feature of our result. Indeed, Strichartz estimates for wave equations are known to fail, in general, for coefficients rougher than $C^{1,1}$, see \cite{st,ss}.\\

Our proof is based on a new approach to Seeger-Sogge-Stein's $L^{p}$ boundedness theorem for FIOs, initiated by Hassell, Rozendaal, and the second author in \cite{HPR}, building on earlier work of Smith \cite{smith}. The approach consists in developing a scale of Hardy spaces $H^{p}_{FIO}$, that are invariant under the action of FIOs. One then shows that this scale relates to the Sobolev scale through the embedding $W^{\frac{s_{p}}{2},p} \subset H^{p}_{FIO} \subset W^{-\frac{s_{p}}{2},p}$, for $p \in (1,\infty)$. This is similar, in spirit, to the theory of Hardy spaces associated with operators, which has been extensively developed over the past 15 years, starting with \cite{amr,hm,dy} (see also the memoir \cite{memoir}).
In this theory, one first constructs a scale of spaces $H^{p} _{ \mathcal{L} }$ by lifting functions from $L^{p}$ to one of the tent spaces introduced by Coifman, Meyer, and Stein in \cite{cms}, using the functional calculus of the operator  $\mathcal{L}$  (rather than convolutions). One then shows that the spaces are invariant under the action of the functional calculus of  $\mathcal{L}$.  Finally, one relates these spaces to more classical ones. For instance $H^{p}_{\Delta}(\R^{d}) = L^{p}(\R^{d})$ for all $p \in (1,\infty)$. More generally, when one considers Hodge-Dirac operators $\Pi_{B}$,  $H^{p}_{\Pi_{B}} = L^{p}$ precisely for those $p$ for which Hodge projections are $L^{p}$ bounded (a result proven by McIntosh and the authors in \cite{FMcP}).\\

In the present paper, we go one step further in connecting both theories, by developing a scale of Hardy-Sobolev spaces $H^{p,s}_{FIO,a}$ on which $\exp(i\sqrt{ \mathcal{L}})$  is bounded, and proving  analogues of the embedding $W^{\frac{s_{p}}{2},p}(\R^{d}) \subset H^{p,0}_{FIO}(\R^{d}) \subset W^{-\frac{s_{p}}{2},p}(\R^{d})$ such as, for $p \in (1,2)$,
$H^{p,\frac{s_{p}}{2}}_{FIO,a} \subset L^{p}$ and $(I+\sqrt{ \mathcal{L} })^{-\frac{s_{p}}{2}} \in B(L^{p}, H^{p,0}_{FIO,a})$. 
This gives our $L^{p}$ boundedness with loss of derivatives result, and more. Indeed, one can apply the half wave group $\exp(i\sqrt{\mathcal{L} })$ repeatedly on $H^{p,s}_{FIO,a}$, and only loose derivatives when one compares $H^{p,s}_{FIO,a}$ to classical Sobolev spaces. This allows for iterative arguments in constructing parametrices (an idea used recently in \cite{HR}). One can also perturb the half wave group using abstract operator theory on the Banach space $H^{p,s}_{FIO,a}$ (see Corollary \ref{cor:pert}). \\

The paper is structured as follows.  In Section \ref{sec:1d}, we treat the problem in dimension 1. In this simple situation, arguments based on bilipschitz changes of variables can be used. 

In Section \ref{sec:transport} we consider the transport group generated, on $L^{2}(\R^{d};\C^{2})$, by 
$$
 i  \xi.D_{a} := \sum \limits _{j=1} ^{d} \xi_{j}\left( \begin{array}{cc} 0 & -  i  a_{j+d}\partial_{j} \\  i  a_{j}\partial_{j} & 0 \end{array} \right), 
$$ 
for $\xi \in \R^{d}$. The dimension 1 results from Section \ref{sec:1d} allow us to prove that  the commuting one dimensional wave groups $(\exp(i t \sqrt{(e_{j}.D_{a})^{2}}))_{t \in \R}$ are bounded in $L^{p}$  for all $p \in [1,\infty)$ and $j=1,...,d$. 
The Phillips functional calculus associated with  the corresponding commutative $d$-parameter group  can then replace convolutions/Fourier multipliers in the context of our Lipschitz metric,  and includes functions of 
$$
L:=D_{a}.D_{a}= \left( \begin{array}{cc} L_{1} & 0 \\ 0 & L_{2} \end{array} \right),
$$
where 
$L_{1}:=- \sum \limits _{j=1} ^{d} 
a_{j+d}\partial_{j}a_{j}\partial_{j}$ and $L_{2}:=- \sum \limits _{j=1} ^{d} 
a_{j}\partial_{j}a_{j+d}\partial_{j}$.
 Using this calculus, we use the approach of \cite{AMcM} to construct an adapted scale of Hardy-Sobolev spaces in Section \ref{sec:HpD}. For all integrability parameters $p \in (1,\infty)$ and regularity parameter $s \in [0,2]$, these spaces coincide with classical Sobolev spaces, thanks to the regularity properties of the heat kernel of  $L$ arising from the Lipschitz continuity of its coefficients. 
To go from these spaces to $H^{p,s}_{FIO,a}$, one needs to directionally refine the Littlewood-Paley decomposition, as in the proof of Seeger-Sogge-Stein's theorem. This is done in \cite{HPR} using a wave packet transform defined by Fourier multipliers. In Section \ref{sec:wavepackets} we construct a similar wave packet transform, replacing Fourier multipliers by the Phillips calculus of the transport group. This allows us to define $H^{p,s}_{FIO,a}$ in Section \ref{sec:space}, and to prove its embedding properties in Section \ref{sec:sob}. In Section \ref{sec:wave}, we prove that the half wave group $(\exp(it\sqrt{L}))_{t \in \R}$ is bounded on $H^{p,s}_{FIO,a}$ for all $1<p<\infty$ and $s\in \R$. To do so, we first notice that the one dimensional wave groups are. We then realise that, in a given direction $\omega$, $\exp(i\sqrt{D_{a}.D_{a}})$ is close to  $\exp(i\sum _{j=1} ^{d} \omega_{j}\sqrt{(e_{j}.D_{a})^{2}})$,  when acting on an appropriate wave packet, in the sense that operators of the form $\big(\exp(i\sqrt{D_{a}.D_{a}})- \exp(i\sum _{j=1} ^{d} \omega_{j}\sqrt{(e_{j}.D_{a})^{2}})  \big) \varphi_{\omega}(D_{a})$ are $L^{p}$ bounded.  Finally, in Section \ref{sec:perturb}, we show that $\exp(it\sqrt{L})$ remains bounded if one appropriately perturbs $L$ by first order terms. This is based on Theorem \ref{thm:perturb}, a result about multiplication operators on $H^{p}_{FIO,a}$ that is of independent interest, even in the case where $a_{j}=1$ for all $j=1,...,2d$. \\

Our approach relies heavily on algebraic properties: the wave group commutes with the wave packet localisation operators, and can be expressed in the Phillips functional calculus of a commutative group. Although our coefficients are merely Lipschitz continuous, these algebraic properties match those of the standard Euclidean wave group. 
However, in dimension $d>1$, the problem does not reduce to its euclidean counterpart through a change of variables (see Remark \ref{rk:notriv}).\\ 

In the same way as 
Peral-Miyachi's result for the standard half wave group is a starting point for the well-posedness theory of wave equations with coefficients that are smooth enough perturbations of constant coefficients, we expect the results proven here to provide a basis for the development of a well-posedness theory of wave equations with coefficients that are smooth enough perturbations of  structured Lipschitz continuous coefficients. 

 \subsection*{Acknowledgments}  
We thank Andrew Hassell and Jan Rozendaal for many interesting discussions on the relations between this work and theirs.  We particularly want to thank Jan Rozendaal for pointing out a mistake in a previous version of Section \ref{sec:perturb}.  We also thank the anonymous referee of a previous version of this paper for pointing out the change of variable approach that we now use in Section \ref{sec:1d}, before moving on to more general operators for which such an approach is not available.

\section{Preliminaries} 
We first recall (a special case of) the following Banach space valued Marcinkiewicz-Lizorkin Fourier multiplier's theorem (see \cite[Theorem 4.5]{sw}).
\begin{Theorem}(Fernandez/ \v{S}trkalj-Weis)
\label{thm:mult}
Let $p \in (1,\infty)$.
Let $m \in C^{1}(\R^{d}\backslash \{0\})$ be such that,
for all $\alpha \in \N_0^d$ with $|\alpha|_\infty\leq 1$ there exists a constant $C=C(\alpha)>0$ such that 
$$
		|\zeta^\alpha \partial_\zeta^\alpha m(\zeta)| \leq C 
\quad \forall \zeta \in \R^d \setminus \{0\}.$$
Let $T_{m}$ denote the Fourier multiplier with symbol $m$.
Then $T_{m} \otimes I_{L^{p}(\R^{d})}$ extends to a bounded operator on $L^{p}(\R^{d};L^{p}(\R^{d}))$.
\end{Theorem}
This theorem will be combined with the following version of the Coifman-Weiss transference principle (see \cite[Theorem 10.7.5]{volume2}). Note that the extension of this theorem from a one parameter group to a $d$ parameter group generated by a tuple of commuting operators is straightforward.
\begin{Theorem}(Coifman-Weiss)
\label{thm:transf}
Let $p \in (1,\infty)$.
Let $iD_{1},...,iD_{d}$ generate bounded commuting groups $(\exp(itD_{j}))_{t \in \R}$ on $L^{p}(\R^{d})$, and consider the $d$ parameter group defined by $\exp(i\xi D) = \prod \limits _{j=1} ^{d} \exp(i\xi_{j}D_{j})$ for $\xi \in \R^{d}$. Then, for all $\psi \in \mathcal{S}(\R^{d})$, we have that 
$$
\| \int \limits _{\R^{d}} \widehat{\psi}(\xi) \exp(i\xi D)f d\xi\|_{L^{p}(\R^{d})}
\lesssim \|T_{\psi}\otimes I_{L^{p}(\R^{d})}\|_{B(L^{p}(\R^{d};L^{p}(\R^{d})))} \|f\|_{L^{p}(\R^{d})} \quad \forall f \in L^{p}(\R^{d}).
$$
\end{Theorem}

To define our Hardy-Sobolev spaces, we use the tent spaces introduced by Coifman, Meyer, and Stein in \cite{cms}, and used extensively in the theory of Hardy spaces associated with operators (see e.g. the memoir \cite{memoir} and the references therein). These tent spaces $T^{p,2}(\R^d)$ are defined as follows.  For $F: \R^d \times (0,\infty) \to \C^{N}$ measurable  and $x \in \R^d$, set
$$
	\calA F(x):= \left( \int_0^\infty \aver{B(x,\sigma)} |F(y,\sigma)|^2 \,dy \frac{d\sigma}{\sigma} \right)^{1/2} \, \in [0,\infty],
$$
where $|\cdot|$ denotes the euclidean norm on $\C^{N}$.
\begin{Def}
Let $p \in [1,\infty)$. The tent space $T^{p,2}(\R^d)$ is defined as the space of all $F \in L^2_{\loc}(\R^d \times (0,\infty),dx\frac{d\sigma}{\sigma})$ such that $\calA F \in L^p(\R^d)$, endowed with the norm 
$$
		\|F\|_{T^{p,2}(\R^d)} := \| \calA F\|_{L^p(\R^d)}. 
$$
\end{Def}

Recall that the tent space $T^{1,2}$ admits an atomic decomposition (see \cite{cms}) in terms of atoms $A$ supported in sets of the form $B(c_{B},r)\times [0,r]$, and satisfying
$$
r^{d}\int \limits _{0} ^{r} \int \limits _{\R^{d}} |A(y,\sigma)|^{2} \frac{dyd\sigma}{\sigma} \leq 1.
$$
Recall also  that the classical Hardy space $H^{1}(\R^{d})$ norm can be obtained as
$$
\|f\|_{H^{1}(\R^{d})} := \|(t,x) \mapsto \psi(t^{2}\Delta)f(x)\|_{T^{1,2}(\R^{d})},
$$
where $\psi(t^{2}\Delta)$ denotes the Fourier multiplier with symbol $\xi \mapsto t^{2}|\xi|^{2} \exp(-t^{2}|\xi|^{2})$.
This is the starting point of the theory of Hardy spaces associated with operators (or equations): one replaces the Fourier multiplier by an appropriately adapted operator. 
To do so, one often uses the holomorphic functional calculus of a (bi)sectorial operator.
The relevant theory is presented in \cite{volume2}. We use it here with the following notation.

\begin{Def}
Let $0<\theta<\frac{\pi}{2}$.  Define the open sector in the complex plane by 
$$
	S_{\theta+}^o :=\{z \in \C \setminus \{0\} \,:\, |\arg(z)|<\theta\},
$$
as well as the bisector $S_{\theta}^o = S_{\theta+}^o\cup S_{\theta-}^o$, where
$S_{\theta-}^o=-S_{\theta+}^o$.
We denote by $H(S_\theta^o)$ the space of holomorphic functions on $S_\theta^o$, and set 
\begin{align*}
	H^\infty(S_\theta^o) &:=\{g \in H(S_\theta^o)\,:\, \|g\|_{L^\infty(S_\theta^o)}<\infty \},\\
	\Psi_\alpha^\beta(S_\theta^0) &:=\{ \psi\in H^\infty(S_\theta^o) \, : \, \exists C>0: |\psi(z)| \leq C |z|^\alpha (1+|z|^{\alpha+\beta})^{-1} \,\forall z \in S_\theta^o\}
\end{align*}
for every $\alpha,\beta>0$.
We say that $\psi \in H^\infty(S_\theta^o)$ is non-degenerate if neither of its restrictions 
to $S_{\theta+}^o$ or $S_{\theta-}^o$ vanishes identically.
\end{Def}

For bisectorial operators $D$ such that $iD$ generates a bounded group on $L^p$, we also use the Phillips calculus defined by
$$
\psi(D)f := \frac{1}{2\pi}  \int \limits _{\R} \widehat{\psi}(\xi) \exp(i\xi D)f d\xi,
$$
for $f \in L^{p}$ and $\psi \in \mathcal{S}(\R)$. See  \cite{AMcM,McM}   for more information on how these two functional calculi interact in the theory of Hardy spaces associated with operators.  The results in Section \ref{sec:HpD} are fundamentally inspired by these papers.

\section{The one dimensional case}
\label{sec:1d}

In dimension one, the type of wave equations we are studying in this paper can be treated through a combination of simple changes of variables and perturbation arguments. In this section, we present this method both for pedagogical reasons, and because its results are used to set up our approach to higher dimensional problems in the next sections. \\

Let $a,b \in C^{0,1}(\R)$ with $\frac{d}{dx} a, \frac{d}{dx} b \in L^{\infty}$, and assume that there exist $0<\lambda\leq \Lambda$ such that $\lambda \leq a(x) \leq \Lambda$ and $\lambda \leq b(x) \leq \Lambda$ for all $x \in \R$. 
We consider the wave equation  $\partial_{t}^{2} u = (a \partial_{x} b \partial_{x})u$. 

\begin{Prop}
The operators  $a \frac{d}{dx}$ and $i\sqrt{-a \frac{d}{dx} a \frac{d}{dx}}$ generate bounded $C_{0}$ groups on $L^{p}(\R)$ for all $p \in (1,\infty)$.
\end{Prop}

\begin{proof}
Define $\phi: x \mapsto \int \limits _{0} ^{x} \frac{1}{a(y)}dy$, and note that it is a $C^{1}$ diffeomorphism from $\R$ onto $\R$. The map $\chi \in C^{1}(\R^{2})$ defined by
$$
\chi: (t,x) \mapsto \phi^{-1}(t+\phi(x)),
$$
is then a solution to $$ \partial_{t} \chi(t,x) = a(\chi (t,x)) \quad \forall t,x \in \R.$$
It is such that
\begin{equation}
\label{eq:imp}
t = \int \limits _{\chi(0,x)} ^{\chi(t,x)} \frac{1}{a_{j}(y)} \, dy \quad \forall t,x \in \R.
\end{equation}
and thus:
$$
\frac{d}{dx}\chi(x,t) = \frac{a(\chi(x,t))}{a(x)} \quad \forall x,t \in \R.
$$
Therefore   $x\mapsto \frac{d}{dx}\chi(x,t)$ is bounded above and below, uniformly in $t$,  and $\chi$ is a thus a bi-Lipschitz flow. We now define the associated transport group by 
$$T_{t}f(x) = f(\chi(t,x)) \quad \forall t,x \in \R$$
for $f \in C_{c}^{\infty}(\R)$. It extends to a bounded group on $L^p(\R^d)$ for all $p \in [1,\infty]$, with finite speed of propagation. Strong continuity $\|T(t)f-f\|_{p} \underset{t \to 0}{\to} 0$ for $p<\infty$ follows by dominated convergence for $f$ continuous, and then density for general $f$. To identify the generator, let $f \in W^{1,p}$, and note that, for all $x \in \R^{d}$,
\begin{align*}
	\frac{\partial}{\partial {t}} T(t)f(x)|_{t=0} 
		&= \frac{\partial}{\partial {t}}f(\chi(x,t))|_{t=0} 
		= \nabla f(x) \cdot \partial_{t} \chi(x,t)|_{t=0} \\
		& =a(x)\partial_{x} f(x).
\end{align*}
For $f \in C_{c}^{\infty}(\R)$, we have that
$$
T_{t}(f \circ \phi)(x) = f(t +\phi(x)) = (\exp(it\frac{d}{dx})f)(\phi(x)) \quad \forall t,x \in \R.
$$
For $f \in C_{c}^{\infty}(\R)$, $s \in \R$, and $\varepsilon>0$, we have that 
$$\exp(-(\varepsilon+is) \sqrt{-a\frac{d}{dx}a\frac{d}{dx}})f
= \frac{1}{\sqrt{2\pi}} \int \limits _{\R} \widehat{\psi_{s}}(t) T_{t}f dt$$
for $\psi_{s}: x\mapsto \exp(-(\varepsilon+is) |x|)$.
We thus have that
$$\exp(-(\varepsilon+is) \sqrt{-a\frac{d}{dx}a\frac{d}{dx}})(f \circ \phi)(x)=
(\exp(-(\varepsilon+is)\frac{d}{dx})f)(\phi(x)) \quad \forall x \in \R,
$$
for all $f \in C_{c}^{\infty}(\R)$, $s \in \R$, and $\varepsilon>0$. On $L^{2}(\R$), $i\sqrt{-a\frac{d}{dx}a\frac{d}{dx}}$ generates a bounded group and $-\sqrt{-a\frac{d}{dx}a\frac{d}{dx}}$ generates an analytic semigroup. We thus have that 
$$\exp(is\sqrt{-a\frac{d}{dx}a\frac{d}{dx}})(f \circ \phi)(x)=
(\exp(is\frac{d}{dx})f)(\phi(x)) \quad \forall x \in \R,
$$
for all $f \in C_{c}^{\infty}(\R)$, and $s \in \R$. Since $\phi$ is a $C^{1}$ diffeomorphism from $\R$ onto $\R$, this gives that $i\sqrt{-a\frac{d}{dx}a\frac{d}{dx}}$ generates a bounded $C_{0}$ group on $L^{p}(\R)$ for all $p \in [1,\infty)$.
\end{proof}

\begin{Cor}
\label{cor:1dwave}
The operators $ i \sqrt{-\frac{d}{dx}a^{2}\frac{d}{dx}}$ and $ i \sqrt{-a\frac{d}{dx}b \frac{d}{dx}}$ generate bounded $C_{0}$ groups on $L^{p}(\R)$ for all $p \in [1,\infty)$.
\end{Cor}

\begin{proof}
We have that $\frac{d}{dx}a^{2}\frac{d}{dx} = a\frac{d}{dx}a\frac{d}{dx}+a'a\frac{d}{dx}$ and
$a\frac{d}{dx}b \frac{d}{dx} = \frac{d}{dx}a b\frac{d}{dx} - a'b\frac{d}{dx}$. 
For all $p \in [1,\infty)$ and all $f \in W^{1,p}(\R)$, we have that $\|a'bf'\|_{p} \leq \|ba'\|_{\infty} \|f'\|_{p}$. The result thus follows from perturbation theory and square root reduction for cosine families, see \cite[Proposition 3.16.3 and Corollary 3.14.13]{ABHN}.
\end{proof}

\section{The transport groups} 
\label{sec:transport}

The method developed in this paper applies to wave equations of the form $\partial_{t} ^{2} u = \sum \limits _{j=1} ^{d} D_{j}^{2} u$. What we need from $D$ is that $iD_{j}$  and $i\sqrt{D_{j}^{2}}$   generates a bounded $C_{0}$ group on $L^{p}$ for each $j$,  the operators $D_{1}^{2},...,D_{d}^{2}$ commute,  and $L=\sum \limits _{j=1} ^{d} D_{j}^{2}$ is such
that appropriate Riesz transform bounds and Hardy space estimates hold. In this section, we consider the simplest non-trivial example of such a Dirac operator. We then use this example throughout the paper, but indicate when the results hold for more general Dirac operators, with the same proofs.\\

For $j \in \{1,\ldots,2d\}$,  let $a_j \in C^{0,1}(\R)$ with $\frac{d}{dx} a_{j} \in L^{\infty}$, and assume that there exist $0<\lambda\leq \Lambda$ such that $\lambda \leq a_j(x) \leq \Lambda$ for all $x \in \R$. We denote by $\widetilde{a_{j}} \in C^{0,1}(\R^{d})$ the map defined by 
$\widetilde{a_{j}}:x \mapsto a_{j}(x_{j})$.

\begin{Def}
\label{def:dirac}
For $\xi = (\xi_{1},...,\xi_{d}) \in \R^{d}$, define 
$$
\xi.D_{a} := \sum \limits _{j=1} ^{d} \xi_{j}\left( \begin{array}{cc} 0 & - \widetilde{a_{j+d}}\partial_{j} \\ \widetilde{a_{j}}\partial_{j} & 0 \end{array} \right), \quad
\xi.\sqrt{D_{a}^{2}} = := \sum \limits _{j=1} ^{d} \xi_{j}\left( \begin{array}{cc} \sqrt{- \widetilde{a_{j+d}}\partial_{j}\widetilde{a_{j}}\partial_{j}}  & 0 \\ 0 & \sqrt{- \widetilde{a_{j}}\partial_{j}\widetilde{a_{j+d}}\partial_{j}}\end{array} \right),
$$
as an unbounded operator acting on $L^{2}(\R^{d};\C^{2})$, with domain $W^{1,2}(\R^{d};\C^{2})$.
\end{Def}
 Note that $W^{1,2}(\R^{d};\C^{2})$ is an appropriate domain for $\xi.\sqrt{D_{a}^{2}}$ thanks to the boundedness of the relevant Riesz transforms proven in \cite[Corollary 5.19]{AMcT}. 

As in \cite[Section 4, Case II]{McM},  $i e_{j}.D_{a}$  generates a bounded $C_{0}$ group on $L^{2}(\R^{d};\C^{2})$  for all $j=1,..,d$, since $e_{j}.D_{a}$  is self-adjoint with respect to an equivalent inner product of the form  $(u,v) \mapsto \langle A^{-1}u,v\rangle$, where $A$ is a  diagonal multiplication operator with $C^{0,1}$ entries.

\begin{Remark} \label{rem:preciseFS}
For $E,F \subset \R^d$ Borel sets and $\omega \in S^{d-1}$, we set $\omega.d(E,F):= \inf_{x \in E, y \in F} |\langle \omega,x-y\rangle|$. By \cite[Remark 3.6]{McM}, we have the 
following (strong) form of finite speed of propagation: there exists $\kappa>0$ such that for all $f \in L^2(\R^d;\C^2)$, all Borel sets $E,F \subset \R^d$, all  $j=1,...,d$,  all $\xi \in \R^{d}$,  and all $\omega \in S^{d-1}$ we have
$$
		1_E \exp(i  \xi_{j} e_{j}.D_a ) (1_F f) =0,
$$
whenever $\frac{\kappa}{\sqrt{d}} |\langle \omega, \xi_{j}e_{j}  \rangle| <\omega.d(E,F)$. Consequently, 
$$
		1_E  \prod _{j=1} ^{d} \exp(i \xi_{j}e_{j}.D_a)  (1_F f) =0,
$$
whenever $\kappa |\langle \omega,\xi \rangle| <\omega.d(E,F)$. Indeed, we have that
$$
1_E  \prod _{j=1} ^{d} \exp(i \xi_{j}e_{j}.D_a)  (1_F f) = 1_E \exp(i \xi_{1}e_{1}.D_a) 1_{E_{1}} \prod _{j=2} ^{d} \exp(i \xi_{j}e_{j}.D_a)  (1_F f),
$$
for $E_{1} = \{(y_{1},x_{2},...,x_{d}) \in \R^{d} \;;\; (x_{1},...,x_{d})\in E \; \text{and} \; |y_{1}-x_{1}| \leq \frac{\kappa}{\sqrt{d}} |\xi_{1}|\}$. Iterating this argument gives us that
$$
1_E  \prod _{j=1} ^{d} \exp(i \xi_{j}e_{j}.D_a)   (1_F f) = 1_E  \prod _{j=1} ^{d} \exp(i \xi_{j}e_{j}.D_a)  1_{\widetilde{E}} (1_F f),
$$
for $\widetilde{E} = \{(y_{1},...,y_{d}) \in \R^{d} \;;\; (x_{1},...,x_{d})\in E \; \text{and} \; |y_{j}-x_{j}| \leq \frac{\kappa}{\sqrt{d}} |\xi_{j}| \quad \forall j=1,...,d\}$.
Assuming that there exists $y \in \widetilde{E} \cap F$ when
$\kappa |\langle \omega,\xi \rangle| <\omega.d(E,F)$, we obtain that, for all $x \in E$,
$$
|\langle \omega,x-y \rangle|\leq  \kappa \underset{j=1,...,d}{\max}|\langle \xi,e_{j} \rangle 
\omega_{j}| < \omega.d(E,F),$$
which is a contradiction.
\end{Remark}

\begin{Prop}
\label{prop:bddgrp}
Let $\xi \in \R^{d}$ and $p \in (1,\infty)$. The group $(\exp(i t\xi.\sqrt{D_{a}^{2}}))_{t \in \R}$ is 
bounded on $L^{p}(\R^{d};\C^{2})$.
\end{Prop}

\begin{proof}
Let $p \in (1,\infty)$.
Using linearity and freezing $d-1$ of the variables, it suffices to show that the group generated by $i \left( \begin{array}{cc} 0 & -b \frac{d}{dx} \\ a\frac{d}{dx} & 0 \end{array} \right)$ is bounded on $L^{p}(\R;\C^{2})$ for $a:=a_{1}$ and $b:=a_{d+1}$. For $f,g \in C^{\infty}_{c}(\R)$, and $t \in \R$, let us consider
$$
\left( \begin{array}{c} u(t,.) \\ v(t,.) \end{array} \right) := \exp\left( i  t\left( \begin{array}{cc} 0 & -b \frac{d}{dx} \\ a\frac{d}{dx} & 0 \end{array} \right)\right) \left( \begin{array}{c} f \\ g \end{array} \right).
$$
We have that
$$
\left( \begin{array}{c} \partial_{t} u(t,.) \\ \partial_{t} v(t,.) \end{array} \right)
=  i 
\left( \begin{array}{c} -b \frac{d}{dx} v(t,.) \\ a \frac{d}{dx}u(t,.) \end{array} \right) \quad \forall t,x \in \R,$$
and
$$
\left( \begin{array}{c} \partial_{t}^{2} u(t,.) \\ \partial_{t}^{2} v(t,.) \end{array} \right)
= 
\left( \begin{array}{c}   -b \frac{d}{dx} a \frac{d}{dx} u(t,.) \\    -a \frac{d}{dx} b \frac{d}{dx}  v(t,.)  \end{array} \right) \quad \forall t,x \in \R.$$
Using Corollary \ref{cor:1dwave} and solving these wave equations using the relevant cosine families (see \cite[Corollary 3.14.12]{ABHN}), this gives
\begin{align*}
\|u(t,.)\| & \lesssim \|f\|_{p} + \| (-b\frac{d}{dx} a \frac{d}{dx})^{-\frac{1}{2}} g' \|_{p} 
\lesssim \|f\|_{p} +\|g\|_{p},
\\
\|v(t,.)\| & \lesssim \|g\|_{p} + \| (-a\frac{d}{dx} b \frac{d}{dx})^{-\frac{1}{2}} f' \|_{p}
\lesssim \|f\|_{p} +\|g\|_{p}, 
\end{align*}
with constants independent of $t$, using
the boundedness of the Riesz transforms $\frac{d}{dx} (-b\frac{d}{dx} a \frac{d}{dx})^{-\frac{1}{2}}$ proven in \cite{AT2,AMN}.

\end{proof}

\begin{Remark}
Given the vector-valued nature of the Dirac operator $D_{a}$, all function spaces considered
in the remaining of the paper will be implicitly $\C^{2}$ valued.
\end{Remark}

\begin{Remark}
\label{rk:notriv}
The transport group generated by $i D_{a}$ is, even in dimension one, substantially more complicated than the transport group generated by $ a\frac{d}{dx}$ considered in Section \ref{sec:1d}. Its $L^p$ boundedness, for instance, does not follow from the boundedness of the translation group through bi-Lipschitz changes of variables. Indeed, for non-constant coefficients $a \in C^{0,1}(\R)$, no intertwining relation 
$$U \left( \begin{array}{cc} 0 & - \frac{d}{dx} \\ a\frac{d}{dx} & 0 \end{array} \right)= \left( \begin{array}{cc} 0 & - \frac{d}{dx} \\ \frac{d}{dx}& 0 \end{array} \right) U$$ can hold for $U$ of the form $U:(f,g) \mapsto (f \circ \phi,g \circ \psi)$ where $\phi, \psi: \R \to \R$ are bi-Lipschitz changes of variables. \end{Remark}

\section{Hardy spaces associated with the transport groups}
\label{sec:HpD}

\begin{Def}
\label{def:calc}
Given $\Psi \in \mathcal{S}(\R^{d})$, we define $\Psi(\sqrt{D_{a}^{2}})$ using the Phillips functional calculus associated with the commutative group $(\exp(i \xi.\sqrt{D_{a}^{2}}))_{\xi \in \R^{d}}$:
$$
\Psi(\sqrt{D_{a}^{2}}) : = \frac{1}{(2\pi)^{d}} \int \limits _{\R^{d}} \widehat{\Psi}(\xi)\exp(i\xi. \sqrt{D_{a}^{2}})d\xi.
$$
We restrict our attention to functions $\Psi$ that satisfy $\Psi=\Psi^{s}$, where 
$$
\Psi^{s}(x) := 2^{-d} \sum \limits _{(\delta_{j})_{j=1} ^{d} \in \{-1,1\}^{d}} \Psi(\delta_{1}x_{1},...,\delta_{d}x_{d}).
$$
For such functions, we have that 
\begin{align*}
\Psi^{s}&(\sqrt{D_{a}^{2}})  \\ 
& = \frac{1}{(2\pi)^{d}} \int \limits _{\R^{d-1}} 
 \int \limits _{\R} 
 \widehat{\Psi^{s}}(\xi)\frac{1}{2}(\exp(i\xi_{1}e_{1}\sqrt{D_{a}^{2}})+\exp(-i\xi_{1}e_{1}\sqrt{D_{a}^{2}}))d\xi_{1}\exp(i(\xi-\xi_{1}e_{1})\sqrt{D_{a}^{2}})d\xi_{2}...,d\xi_{d} \\
 & = \frac{1}{(2\pi)^{d}} \int \limits _{\R^{d-1}} 
 \int \limits _{\R} 
 \widehat{\Psi^{s}}(\xi)\frac{1}{2}(\exp(i\xi_{1}e_{1}D_{a})+\exp(-i\xi_{1}e_{1}D_{a}))d\xi_{1}\exp(i(\xi-\xi_{1}e_{1})\sqrt{D_{a}^{2}})d\xi_{2}...,d\xi_{d} \\
 &= \frac{1}{(2\pi)^{d}} \int \limits _{\R^{d}} \widehat{\Psi^{s}}(\xi) \prod _{j=1} ^{d} \exp(i \xi_{j}e_{j}.D_a)  d\xi,
\end{align*}
since $e_{j}.D_{a}$ and $e_{j}.\sqrt{D_{a}^{2}}$ generate the same cosine family. We write $\Psi(D_{a})$ instead of $\Psi(\sqrt{D_{a}^{2}})$ when $\Psi=\Psi^{s}$.
\end{Def}

\begin{Lemma}
\label{lem:od}
There exists $C>0$ such that, for all $\Psi \in \mathcal{S}(\R^{d})$  such that $\Psi=\Psi^{s}$, all $E,F \subset \R^{d}$ Borel sets and all $\omega \in S^{d-1}$,  we have that
$$
\|1_{E} \Psi(D_{a})(1_{F}f)\|_{2} \leq C \|1_{F}f\|_{2} \int \limits _{ \{ |\xi| \geq \frac{d(E,F)}{\kappa} \} \cap \{|\langle \omega,\xi\rangle| \geq \frac{\omega.d(E,F)}{\kappa} \} }  |\widehat{\Psi}(\xi)| d\xi  \quad \forall f \in L^{2}(\R^{d}).
$$
Consequently, for every  $\Psi \in \mathcal{S}(\R^{d})$ and every  $M \in \N$, there exists $C_{M}>0$ such that
$$
\|1_{E} \Psi(\sigma D_{a})(1_{F}f)\|_{2} \leq C_{M} (1+\frac{d(E,F)}{\kappa \sigma})^{-M}\|1_{F}f\|_{2}   \quad \forall f \in L^{2}(\R^{d})
$$
for all Borel sets $E,F \subset \R^{d}$ and all $\sigma >0$.
\end{Lemma}

\begin{proof}
Let $f \in L^{2}(\R^{d})$ and $\xi \in \R^{d}$. By Remark \ref{rem:preciseFS}, we have that
$$
1_{E}  \prod _{j=1} ^{d} \exp(i \xi_{j}e_{j}.D_a)  (1_{F}f) = 0,
$$
whenever $\kappa |\xi| < d(E,F)$ or $\kappa |\langle \omega,\xi\rangle| <\omega.d(E,F)$.  Therefore, using Phillips functional calculus, we have that
\begin{align*}
\|1_{E} \Psi(D_{a})(1_{F}f)\|_{2}  
& \leq \frac{1}{(2\pi)^d} \int \limits _{\R^{d}} |\widehat{\Psi}(\xi)| \|1_{E} \prod _{j=1} ^{d} \exp(i \xi_{j}e_{j}.D_a)  (1_{F}f)\|_{2} d\xi \\
& \leq C \|1_{F}f\|_{2} \int \limits _{ \{ |\xi| \geq \frac{d(E,F)}{\kappa} \} \cap \{|\langle \omega,\xi\rangle| \geq \frac{\omega.d(E,F)}{\kappa} \} } |\widehat{\Psi}(\xi)| d\xi,
\end{align*}
where $C:=\frac{1}{(2\pi)^d}  \sup \{\| \prod _{j=1} ^{d} \exp(i \xi_{j}e_{j}.D_a)  \|_{B(L^2)} \;;\; \xi \in \R^{d}\}$. The last statement then follows from a change of variables and $\Psi \in \mathcal{S}(\R^{d})$.
\end{proof}

We recall the following fact, which is a corollary of the results in \cite{AMcT}, using that the coefficients $a_{j}$ are Lipschitz continuous.
\begin{Theorem}(Auscher, McIntosh, Tchamitchian)
\label{thm:AMcT}
Let $p \in (1,\infty)$. On $L^{p}(\R^{d})$, 
the operator $L= D_{a}^{2}$,  with domain $W^{2,p}(\R^d)$, generates an analytic semigroup, and has a bounded $H^{\infty}$ calculus of angle $0$.
Moreover, $\{\exp(-tL) \;;\; t>0\}$ satisfies Gaussian estimates.
\end{Theorem}

\begin{Cor}
\label{cor:funcalc}
Let $p \in (1,\infty)$, $\theta>0$, $g \in H^\infty(S_{\theta+}^o)$,  and let  $\Psi \in C_c^\infty(\R^d)$ be supported away from $0$ and such that $\Psi=\Psi^{s}$.  Then there exists a constant $C>0$  independent of $g$  such that, for all $F \in T^{p,2}(\R^{d})$, 
$$\|(\sigma,x) \mapsto \Psi(\sigma D_{a})g(L)F(\sigma,.)(x) \|_{T^{p,2}(\R^{d})} \leq 
 C \|g\|_{L^\infty(S_{\theta+}^o)} \|(\sigma,x) \mapsto F(\sigma,.)(x) \|_{T^{p,2}(\R^{d})}.
$$ 
\end{Cor}

\begin{proof}
 
For $M \in \N$, set $q_M(z):= z^{M} (1+z)^{-2M}$, $z \in S_{\theta+}^o$. Note that then $q_M \in \Psi_M^M(S_{\theta+}^o)$.  The statement for $\Psi(\sigma D_a)$ replaced by $q_M(\sqrt{\sigma} L)$ for $M$ large enough then 
follows from a combination of \cite[Theorem 5.2]{HvNP} and \cite[Lemma 7.3]{HvNP}, using Lemma \ref{lem:od} and Theorem \ref{thm:AMcT} to check the assumptions. 
 
 On the other hand, we have by assumption $\zeta \mapsto \Psi(\zeta) q_M^{-1}(|\zeta|^2) \in \calS(\R^d)$, so that an application of \cite[Theorem 5.2]{HvNP} together with Lemma \ref{lem:od} yields the assertion. 
\end{proof}

\begin{Lemma}
\label{lem:sfequiv}
Let  $\alpha \in \R$, and  non-degenerate   $\Psi,\widetilde{\Psi} \in C_{c} ^{\infty} (\R^{d})$ be supported away from $0$ and such that $\Psi=\Psi^{s}$, $\widetilde{\Psi}=\widetilde{\Psi}^{s}$. Let $p \in [1,\infty)$. Then
 $$\|(\sigma,x) \mapsto  \sigma^{\alpha} \Psi(\sigma D_{a})f(x) \|_{T^{p,2}(\R^{d})}  \sim
\|(\sigma,x) \mapsto  \sigma^{\alpha}  \widetilde{\Psi}(\sigma D_{a})f(x) \|_{T^{p,2}(\R^{d})},$$
for all $f$ such that the above quantities are finite.
Moreover, for 
$L= - D_{a}^{2}$, we have that
$$\|(\sigma,x) \mapsto \Psi(\sigma D_{a})f(x) \|_{T^{p,2}(\R^{d})}  \sim
\|(\sigma,x) \mapsto \sigma ^{2} L \exp(-\sigma^{2}L)f(x) \|_{T^{p,2}(\R^{d})}.$$
\end{Lemma}

\begin{proof}
Since $$\|(\sigma,x) \mapsto  \sigma^{\alpha} \Psi(\sigma D_{a})f(x) \|_{T^{p,2}(\R^{d})} \sim
\|(\sigma,x) \mapsto \int \limits _{0} ^{\infty}  \sigma^{\alpha} \Psi(\sigma D_{a}) (\widetilde{\Psi})^{ 2 }(\tau D_{a}) f(x) \frac{d\tau}{\tau} \|_{T^{p,2}(\R^{d})},$$
by  \cite[Corollary 5.1]{HvNP},  it suffices to show that, for all $\sigma, \tau>0$,  $(\frac{\sigma}{\tau})^{\alpha} \Psi(\sigma D_{a}) \widetilde{\Psi}(\tau D_{a}) = \min(\frac{\sigma}{\tau},\frac{\tau}{\sigma})^N S_{\sigma,\tau}$ for some $N>\frac{d}{2}$ and  a family of operators $S_{\sigma, \tau} \in B(L^{2})$ such that for every $M \in \N$, there exists $C_{M}>0$ such that
$$
\|1_{E} S_{\sigma, \tau}(1_{F}f)\|_{2} \leq C_{M} (1+\frac{d(E,F)}{\kappa \max(\sigma,\tau)})^{-M}\|1_{F}f\|_{2}   \quad \forall f \in L^{2}(\R^{d})
$$
for all Borel sets $E,F \subset \R^{d}$ and all $\sigma >0$. This follows from Lemma \ref{lem:od} using that, for all $\xi \in \R^{d}\backslash\{0\}$,  
$$
 (\frac{\sigma}{\tau})^{\alpha}  \Psi(\sigma \xi)\widetilde{\Psi}(\tau \xi) = (\frac{\sigma}{\tau})^{N'- \alpha } \overline{\Psi}(\sigma \xi) \underline{\widetilde{\Psi}}(\tau \xi) = (\frac{\tau}{\sigma})^{N'+ \alpha } \underline{\Psi}(\sigma \xi) \overline{\widetilde{\Psi}}(\tau \xi),
$$ 
for $\overline{\Psi}:\xi \mapsto \frac{\Psi (\xi)}{\xi^\beta}$ and $\underline{\Psi}:\xi \mapsto \xi^\beta \Psi(\xi)$ with $\beta \in \N^d$, $|\beta|_1 =N'$, for $N'>|\alpha|+N$. 
For the second statement,  we first show the comparison of  $\Psi(\sigma D_{a})$ with $(\sigma ^{2} L)^M \exp(-\sigma^{2}L)$ for some $M \in \N$, $M>\frac{d}{4}$ in the exact same way as above. For the comparison of $(\sigma ^{2} L)^M \exp(-\sigma^{2}L)$ with $\sigma ^{2} L \exp(-\sigma^{2}L)$, we use  \cite[Proposition 10.1]{FMcP} instead of \cite[Corollary 5.1]{HvNP}, together with the Gaussian estimates for $\exp(-tL)$ as stated in Theorem \ref{thm:AMcT}.  
\end{proof}

\begin{Theorem}
\label{thm:HpD}
 Let $s \in \R$, let $p \in (1,\infty)$. 
For all  non-degenerate  $\Psi \in C_{c} ^{\infty} (\R^{d})$ supported away from $0$  such that $\Psi=\Psi^{s}$, and all $M \in \N$, we have that
\begin{equation}
\label{eq:wTp}
\|(\sigma,x) \mapsto 1_{[0,1)}(\sigma)\sigma^{-s}\Psi(\sigma D_{a})f(x) 
+1_{[1,\infty)}(\sigma)\Psi(\sigma D_{a})f(x) \|_{T^{p,2}(\R^{d})} \sim
\|(I+\sqrt{L})^{s}f\|_{p}, 
\end{equation}
for all $f \in D((I+\sqrt{L})^{s})$.
Moreover, for $s \in [0,2]$, we have that
\begin{equation}
\label{eq:sob}
\|(\sigma,x) \mapsto 1_{[0,1)}(\sigma)\sigma^{-s}\Psi(\sigma D_{a})f(x) 
+1_{[1,\infty)}(\sigma)\Psi(\sigma D_{a})f(x) \|_{T^{p,2}(\R^{d})} \sim \|f\|_{W^{s,p}} 
\end{equation}
for all $f \in W^{s,p}(\R^{d})$.
\end{Theorem}
\begin{proof}
We use the Hardy space $H^{p}_{L}$ associated with $L$, as defined in \cite{DuongLi}. 
For all $f \in L^{p} \cap L^{2}$, we have, by Lemma \ref{lem:sfequiv},
$$
\|(\sigma,x) \mapsto \Psi(\sigma D_{a})f(x) \|_{T^{p,2}(\R^{d})}  \sim \|f\|_{H^{p}_{L}}.
$$
 It is a folklore fact that $H^{p}_{L} = L^{p}$ for $p \in (1,\infty)$, thanks to the heat kernel bounds of $(e^{tL})_{t\geq 0}$. This result appeared in draft form in an unpublished manuscript of Auscher, Duong, McIntosh, and inspired the proofs of many similar results. For our particular $L$, an appropriate version of the result does not seem to have appeared in the literature. It can however be proven as follows. By \cite[Theorem 4.19]{AMcT}, the operators $tL\exp(-tL)$ have standard kernels satisfying the assumptions of \cite[Theorem 4.4]{htv}. Therefore, for all $f \in L^{p} \cap L^{2}$, $f \in H^{p}_{L}$ and
$$
\|f\|_{H^{p}_{L}} \lesssim \|f\|_{p}.
$$
 The reverse inequality is proven in \cite[Proposition 4.2]{DuongLi} for $p \leq 2$. Given that the above reasoning also applies to $L^{*}$, we obtain the full result by duality.
Combined with Lemma \ref{lem:sfequiv}, this gives the result for $s=0$.
For $s \in \N$, using Lemma \ref{lem:sfequiv} with an appropriate $\widetilde{\Psi} \in C_{c} ^{\infty} (\R^{d})$, we then have that
\begin{align*}
\|(\sigma,x) \mapsto 1_{[0,1)}(\sigma)\sigma^{-s}\Psi(\sigma D_{a})f(x)\|_{T^{p,2}(\R^{d})}
&\lesssim \|(\sigma,x) \mapsto 1_{[0,1)}(\sigma)\widetilde{\Psi}(\sigma D_{a})L^{\frac{s}{2}}f(x)\|_{T^{p,2}(\R^{d})}\\
&\lesssim \|L^{\frac{s}{2}}f\|_{p} \lesssim \|(I+\sqrt{L})^{s}f\|_{p}.
\end{align*}
We also have that
\begin{align*}
\|(\sigma,x) \mapsto 1_{[1,\infty)}(\sigma)\Psi(\sigma D_{a})f(x)\|_{T^{p,2}(\R^{d})}
\lesssim \|f\|_{p} \lesssim \|(I+\sqrt{L})^{s}f\|_{p}.
\end{align*}
For $-s \in \N$, we have that 
\begin{align*}
&\|(\sigma,x) \mapsto 1_{[0,1)}(\sigma)\sigma^{-s}\Psi(\sigma D_{a})f(x)\|_{T^{p,2}(\R^{d})} \\
&\qquad \lesssim \sum \limits _{k=0} ^{|s|} 
\|(\sigma,x) \mapsto 1_{[0,1)}(\sigma)\sigma^{|s|}L^{\frac{k}{2}}\Psi(\sigma D_{a})(I+\sqrt{L})^{-|s|}f(x)\|_{T^{p,2}(\R^{d})}\\
&\qquad \lesssim \sum \limits _{k=0} ^{|s|} 
\|(\sigma,x) \mapsto 1_{[0,1)}(\sigma)\widetilde{\Psi}(\sigma D_{a})(I+\sqrt{L})^{-|s|} f(x)\|_{T^{p,2}(\R^{d})}\lesssim \|(I+\sqrt{L})^{s}f\|_{p},
\end{align*}
as well as
\begin{align*}
&\|(\sigma,x) \mapsto 1_{[1,\infty)}(\sigma)\Psi(\sigma D_{a})f(x)\|_{T^{p,2}(\R^{d})} \\
&\qquad \lesssim \sum \limits _{k=0} ^{|s|} 
\|(\sigma,x) \mapsto 1_{[1,\infty)}(\sigma)\sigma^{k}L^{\frac{k}{2}}\Psi(\sigma D_{a})(I+\sqrt{L})^{-|s|}f(x)\|_{T^{p,2}(\R^{d})}\\
&\qquad \lesssim \sum \limits _{k=0} ^{|s|} 
\|(\sigma,x) \mapsto 1_{[0,1)}(\sigma)\widetilde{\Psi}(\sigma D_{a})(I+\sqrt{L})^{-|s|} f(x)\|_{T^{p,2}(\R^{d})}\lesssim \|(I+\sqrt{L})^{s}f\|_{p}.
\end{align*}
Reverse inequalities are proven similarly, using that, for all $s \in \R$,
$$
\|(I+\sqrt{L})^{s}f\|_{p} \sim \|(\sigma,x) \mapsto  (I+\sqrt{L})^{s}  \Psi(\sigma D_{a})f(x)\|_{T^{p,2}(\R^{d})}.
$$
This gives \eqref{eq:wTp} for all $s \in \Z$, and the result for all $s \in \R$ then follows by complex interpolation of weighted tent spaces as in \cite[Theorem 2.1]{amenta}.\\
To obtain \eqref{eq:sob} one first remarks that, for $s \in \{0,1,2\}$, the above reasoning also gives
$$
\|(\sigma,x) \mapsto 1_{[0,1)}(\sigma)\sigma^{-s}\Psi(\sigma D_{a})f(x) 
+1_{[1,\infty)}(\sigma)\Psi(\sigma D_{a})f(x) \|_{T^{p,2}(\R^{d})} \sim
\sum \limits _{m=0} ^{s}  \|D_{a}^{m}f\|_{p}, 
$$
for all $f \in \bigcap \limits _{m=0} ^{s} D\big(D_{a}^{m}\big)$. We then notice that, for all $j=1,...,d$, we have that
$
\|\partial_{j} f\|_{p} \sim \|\widetilde{a_{j}}\partial _{j}f\|_{p} \sim \|\widetilde{a_{j+d}}\partial _{j}f\|_{p}$, 
and thus $\|f\|_{W^{1,p}} \sim \|f\|_{p} + \|D_{a}f\|_{p}$, for all $f \in W^{1,p}$.
Moreover, 
$$
\widetilde{a_{j+d}} \partial_{j}\widetilde{a_{j}}\partial _{j}f = \widetilde{a_{j}'}\widetilde{a_{j+d}}\partial _{j}f + \widetilde{a_{j}}\widetilde{a_{j+d}}\partial_{j}^{2}f \quad \forall f \in W^{2,p},
$$
 and thus
$$
\|f\|_{W^{2,p}} \sim  \|f\|_{p} +  \|D_{a}f\|_{p} + \|D_{a}^{2}f\|_{p} \quad \forall f \in W^{2,p}.
$$
\end{proof}

\begin{Cor}
\label{cor:sobemb}
Let $\alpha\geq 0$, $p \in (1,\infty)$, and $q \in [p,\infty)$ be such that
$$
 \alpha = \frac{d}{2} (\frac{1}{p}-\frac{1}{q}).
$$
Then there exists $C>0$ such that, for all $f \in L^{p}(\R^{d})$ with $L^{\alpha}f \in L^{p}(\R^{d})$, we have that
$$
\|f\|_{L^{q}(\R^{d})} \leq C \|L^{\alpha}f\|_{L^{p}(\R^{d})}.
$$
\end{Cor}

\begin{proof}
For $f \in L^{p}(\R^{d})$ with $L^{\alpha}f \in L^{p}(\R^{d})$, 
Theorem \ref{thm:HpD} gives that
\begin{align*}
&\|f\|_{L^{q}(\R^{d})}
\lesssim \|(\sigma, x) \mapsto L^{-\alpha}\Psi(\sigma D_{a})L^{\alpha}f(x)\|_{T^{q,2}(\R^{d})}\\
& \quad \lesssim 
\|(\sigma, x) \mapsto \sigma^{2\alpha}\widetilde{\Psi}(\sigma D_{a})L^{\alpha}f(x)\|_{T^{q,2}(\R^{d})}
\end{align*}
for $\widetilde{\Psi}:\xi \mapsto |\xi|^{-\alpha} \Psi(\xi)$.
Using the embedding properties of weighted tent spaces proven in \cite[Theorem 2.19]{amenta}, we have that
$$
\|(\sigma, x) \mapsto \sigma^{2\alpha}\widetilde{\Psi}(\sigma D_{a})L^{\alpha}f\|_{T^{q,2}(\R^{d})}
\lesssim \|(\sigma, x) \mapsto \widetilde{\Psi}(\sigma D_{a})L^{\alpha}f\|_{T^{p,2}(\R^{d})},
$$
and thus
$$
\|f\|_{L^{q}(\R^{d})} \lesssim \|L^{\alpha}f\|_{L^{p}(\R^{d})},
$$
by Theorem \ref{thm:HpD}.

\end{proof}

\begin{Remark}
All results in this section, except \eqref{eq:sob}, hold for a general Dirac operator $D_{a}$
 such that $i e_{j}D_{a}$ and $i e_{j}\sqrt{D_{a}^{2}}$ generate bounded $C_{0}$ group on $L^{p}$ for each $j$, the operators $D_{1}^{2},...,D_{d}^{2}$ commute, $(\exp(it\xi.D_{a}))_{t \in \R}$ has finite speed  of propagation as in Remark \ref{rem:preciseFS}, and  $H^{p}_{D_{a}^{2}} = L^{p}$.  Property \eqref{eq:sob} also holds as long as $D(D_{a})=W^{1,p}$ and $D(D_{a}^{2}) = W^{2,p}$ with equivalence of norms. All results in the next sections also hold for such Dirac operators.
\end{Remark}
\section{Wave packet transform} 
\label{sec:wavepackets}

We use a wave packet transform which is similar to the ones used in \cite{HPR, Rozendaal}, but symmetrised to ensure $\Psi_{\omega,\sigma} = \Psi_{\omega,\sigma}^{s}$.  
\\

Let $\Psi \in C_c^{\infty}(\R^d)$ be a non-negative radial function with $\Psi(\zeta)=0$ for $|\zeta| \notin [\frac12,2]$, and 
\begin{equation} \label{def:Psi}
		\int_0^\infty \Psi(\sigma \zeta)^2 \,\frac{d\sigma}{\sigma} = 1
\end{equation}
for $\zeta \neq 0$. 
Let $\varphi \in C_c^\infty(\R^d)$ be a radial, non-negative function with $\varphi(\zeta)=1$ for $|\zeta|\leq \frac12$ and $\varphi(\zeta)=0$ for $|\zeta|>1$. 
These functions $\Psi, \varphi$ are now fixed for the remainder of the paper. 

For $\omega \in S^{d-1}$, $\sigma>0$ and $\zeta \in \R^d \setminus \{0\}$, set  $\underline{\varphi_{\omega,\sigma}}(\zeta):=c_\sigma \varphi\left(\frac{\hat \zeta-\omega}{\sqrt{\sigma}}\right)$, and $\varphi_{\omega,\sigma} = \underline{\varphi_{\omega,\sigma}}^{s}$, 
where $\displaystyle c_\sigma:=\left(\int_{S^{d-1}} \varphi\left(\frac{e_1-\nu}{\sqrt{\sigma}}\right)^2\,d\nu\right)^{-1/2}$. Set $\varphi_{\omega,\sigma}(0):=0$. 
Set furthermore $\Psi_\sigma(\zeta):=\Psi(\sigma \zeta)$ and $\psi_{\omega,\sigma}(\zeta):=\Psi_\sigma(\zeta)\varphi_{\omega,\sigma}(\zeta)$ for $\omega \in S^{d-1}$,  $\sigma>0$ and $\zeta \in \R^d$. 
By construction, we then have
\begin{align} \label{eq:psi-identity}
	\int_0^\infty \int_{S^{d-1}} \psi_{\omega,\sigma}(\zeta)^2 \,d\omega \frac{d\sigma}{\sigma}=1
\end{align}
for all $\zeta \in \R^d \setminus \{0\}$, see \cite[Lemma 4.1]{HPR}.
For $\omega \in S^{d-1}$ and $\zeta \in \R^d$, we moreover set 
$$
		\varphi_\omega(\zeta):= \int_0^4 \psi_{\omega,\tau}(\zeta)\,\frac{d\tau}{\tau}.
$$

For the convenience of the reader, we recall the following properties of $\psi_{\omega,\sigma}$ stated in
 \cite[Lemma 3.2]{Rozendaal}.  Note that the symmetrisation (using $\varphi_{\omega,\sigma}$ instead of $\underline{\varphi_{\omega,\sigma}}$) only affects \eqref{eq:supptheta}. See also Remark \ref{rk:support} below. 
 
\begin{Lemma} \label{lem:wavepackprop}
Let $\omega \in S^{d-1}$ and $\sigma \in (0,1)$. Each $\zeta \in \supp(\psi_{\omega,\sigma})$ satisfies 
\begin{equation} \label{eq:supptheta}
	\frac{1}{2\sigma} \leq |\zeta| \leq \frac{2}{\sigma}, \qquad 
	 \underset{(\varepsilon_{j})_{j=1} ^{d} \in \{-1,1\}^{d}}{\min}
	|(\varepsilon_{1}\hat \zeta_{1},...,\varepsilon_{d}\hat \zeta_{d}) - \omega|  \leq 2 \sqrt{\sigma}. 
\end{equation}
For all $\alpha \in \N_0^d$ and $\beta \in \N_0$ there exists a constant $C=C(\alpha,\beta) >0$ such that 
\begin{equation} \label{eq:sizetheta}
	|\langle \omega,\nabla_\zeta\rangle^\beta \partial_\zeta^\alpha \psi_{\omega,\sigma}(\zeta)|
	\leq C \sigma^{-\frac{d-1}{4}+\frac{|\alpha|_1}{2}+\beta}
\end{equation}
for all $(\zeta,\omega,\sigma) 
\in \R^d \times S^{d-1} \times (0,\infty)$. 
For every $N \geq 0$ there exists a constant $C_N >0$ such that 
\begin{equation} \label{eq:sizeFourier-theta}
	|\calF^{-1}(\psi_{\omega,\sigma})(x)|
		\leq C_N \sigma^{-\frac{3d+1}{4}} (1+\sigma^{-1}|x|^2 + \sigma^{-2}\langle \omega,x\rangle^2)^{-N}
\end{equation}
for all $(x,\omega,\sigma) \in \R^{d}\times S^{d-1}\times (0,\infty)$. \\
 In particular, $\{\sigma^{\frac{d-1}{4}} \calF^{-1}(\psi_{\omega,\sigma})\,|\,\omega \in S^{d-1}, \,\sigma>0\} \subseteq L^1(\R^d)$ is uniformly bounded. 
\end{Lemma}

We also recall important properties of the family  $(\varphi_\omega)_{\omega \in S^{d-1}}$ from \cite[Remark 3.3]{Rozendaal}.

\begin{Lemma}
Let $\omega \in S^{d-1}$. By construction, $\varphi_\omega \in C^\infty(\R^d)$, and for $\zeta \neq 0$, $\varphi_\omega(\zeta)=0$ for $|\zeta|<\frac{1}{8}$ or $\underset{(\varepsilon_{j})_{j=1} ^{d} \in \{-1,1\}^{d}}{\min}
	|(\varepsilon_{1}\hat \zeta_{1},...,\varepsilon_{d}\hat \zeta_{d}) - \omega|  >2|\zeta|^{-1/2}$. Moreover, for all $\alpha \in \N_0^d$ and $\beta \in \N_0$, there exists a constant $C=C(\alpha,\beta)>0$ such that 
$$
		|\langle \omega,\nabla_\zeta\rangle^\beta \partial_\zeta^\alpha \varphi_\omega(\zeta)| \leq C|\zeta|^{\frac{d-1}{4}-\frac{|\alpha|_1}{2}-\beta}
$$
for all $\omega \in S^{d-1}$ and $\zeta \neq 0$, and 
\begin{equation} \label{eq:cond-psi} 
	| \langle \hat \zeta,\nabla_\zeta\rangle^\beta \partial_\zeta^\alpha \left(\int_{S^{d-1}} \varphi_\nu(\zeta)^2 \,d\nu\right)| \leq C |\zeta|^{-\frac{|\alpha|_1}{2} -\beta}
\end{equation} 
for all $\zeta \in \R^d\setminus \{0\}$. 
\end{Lemma}

\begin{Remark}
\label{rk:support}
For $\omega=e_1$ and  $\zeta$, $\sigma$ chosen as in \eqref{eq:supptheta} with $\sigma \in (0,2^{-8})$, we have 
\begin{align} \label{eq:zetacomp}
	\frac{1}{4\sigma} < |\zeta_1|  \leq \frac{2}{\sigma}, \qquad |\zeta_j|\leq \frac{4}{\sqrt{\sigma}}, \qquad j\in \{2,\ldots,d\}.
\end{align}
This follows from 
\begin{align*}
	|(\varepsilon_{1}\hat \zeta_{1},...,\varepsilon_{d}\hat \zeta_{d}) - e_{1}|^{2} 
	&= |e_1 . ((\varepsilon_{1}\hat \zeta_{1},...,\varepsilon_{d}\hat \zeta_{d}) -e_1)|^2 + \sum_{j=2}^d |e_j.((\varepsilon_{1}\hat \zeta_{1},...,\varepsilon_{d}\hat \zeta_{d}) -e_1)|^2\\
	&= |\frac{\varepsilon_{1}\zeta_1}{|\zeta|}-1|^2 + \sum_{j=2}^d|\frac{\zeta_j}{|\zeta|}|^2, 
\end{align*}
for all $(\varepsilon_{j})_{j=1} ^{d} \in \{-1,1\}^{d}$.  Therefore we have that, for some $\varepsilon_{1} \in \{-1,1\}$, 
\begin{align*}
	|\varepsilon_{1}\zeta_1 -|\zeta||^2 + \sum_{j=2}^d |\zeta_j|^2 \leq 4\sigma|\zeta|^2 \leq \frac{16}{\sigma},
\end{align*}
 which directly yields \eqref{eq:zetacomp} for $j \geq 2$. The case $j=1$ then follows from
\begin{align*}
  |\zeta_1|= |\varepsilon_{1}\zeta_1| >|\zeta|-\frac{4}{\sqrt{\sigma}}
	\geq \frac{1}{2\sigma} - \frac{4}{\sqrt{\sigma}}.
\end{align*} 
\end{Remark}

\begin{Lemma}
\label{lem:repro}
For all $\sigma \in (0,1)$,  and all $f \in L^{2}(\R^{d})$, we have that
\begin{align}
 \label{eq:resol-identity}
		  |S^{d-1}|^{-1}  \int_{S^{d-1}} \int_1^\infty \Psi(\sigma D_{a})^{2} f \,\frac{d\sigma}{\sigma} d\omega 
		 \quad +  \int_{S^{d-1}} \int_0^1 \varphi_{\omega}(D_{a})^{2}\Psi(\sigma D_{a})^{2} f \,\frac{d\sigma}{\sigma} d\omega = f
\end{align}

\begin{equation} \label{eq:reproL2}
\int \limits _{S^{d-1}} \varphi_{\omega, \sigma}(D_{a})^{2}f \, d\omega = f,
\end{equation}

\begin{equation} \label{eq:repro}
\sigma^{-\frac{d-1}{4}} \int \limits _{S^{d-1}} \varphi_{\omega, \sigma}(D_{a})f \, d\omega = C_{\sigma}f,
\end{equation}
with constant $C_{\sigma}$ such that $\sigma \mapsto C_{\sigma}$ is bounded above and below. \end{Lemma} 

\begin{proof}
These identities follow (respectively) from \eqref{eq:psi-identity}, the fact that $\int \limits _{S^{d-1}}  \varphi_{\omega,\sigma}(\xi)^{2} d\omega = 1$ for all $\xi \neq 0$, and \cite[Formula (7.4)]{HPR}, using the Philipps functional calculus of  $\sqrt{D_{a}^{2}}$.\end{proof}

\begin{Lemma}
\label{cor:sim}
For all $\sigma \in (0,1)$, we have that 
$$
\int \limits _{S^{d-1}} \|\varphi_{\omega,\sigma}(D_{a}) f\|_{2} ^{2} \, d\omega 
\lesssim \|f\|_{2} ^{2} \quad \forall f \in L^{2}(\R^{d}).
$$
Moreover,
$$
\int \limits _{S^{d-1}} \int \limits _{0} ^{\infty} \|\psi_{\omega,\sigma}(D_{a}) f\|_{2} ^{2} \,  \frac{d\sigma}{\sigma} d\omega 
\lesssim \|f\|_{2} ^{2} \quad \forall f \in L^{2}(\R^{d}).
$$
\end{Lemma}

\begin{proof}
Let $f \in L^{2}(\R^{d})$ and $\sigma \in (0,1)$. 
Using \eqref{eq:reproL2},  and the fact that  $\sqrt{D_{a}^{2}}$  is self-adjoint with respect to an equivalent inner product (see Definition \ref{def:dirac}),   we have that
\begin{align*}
\int \limits _{S^{d-1}} \|\varphi_{\omega,\sigma}(D_{a}) f\|_{2} ^{2} \, d\omega 
 \sim  \int \limits _{S^{d-1}} \langle \varphi_{\omega,\sigma}(D_{a})^{2} f,f \rangle \, d\omega 
 \lesssim  \|f\|_{2} ^{2}.
\end{align*}
Similarly, using \eqref{eq:resol-identity}, we have that
$$
\int \limits _{S^{d-1}} \int \limits _{0} ^{\infty} \|\psi_{\omega,\sigma}(D_{a}) f\|_{2} ^{2} \,  \frac{d\sigma}{\sigma} d\omega
 \sim  \int \limits _{S^{d-1}} \int \limits _{0} ^{\infty} \langle \psi_{\omega,\sigma}(D_{a})^{2} f,f\rangle \,  \frac{d\sigma}{\sigma} d\omega \lesssim \|f\|_{2} ^{2}.
$$
\end{proof}

\begin{Def}
We define a wave packet transform adapted to $D_{a}$,\\ $W_{a} \in B(L^{2}(\R^{d}),L^{2}(\R^{d}\times S^{d-1}\times(0,\infty); dx d\omega \frac{d\sigma}{\sigma}))$ by
$$
W_{a}f(\omega, \sigma, x):= 1_{(1,\infty)}(\sigma)   |S^{d-1}|^{-1/2}  \Psi(\sigma D_{a})f(x) 
+ 1_{[0,1]}(\sigma)\varphi_{\omega}(D_{a})\Psi(\sigma D_{a})f(x) \quad \forall f \in L^{2}(\R^{d}).
$$
 
We define $\pi_a \in B(L^{2}(\R^{d}\times S^{d-1}\times(0,\infty);dx d\omega \frac{d\sigma}{\sigma}),L^{2}(\R^{d}))$ by 
\begin{align*}
		\pi_a F(x) := & |S^{d-1}|^{-1/2}  \int_{S^{d-1}} \int_1^\infty \Psi(\sigma D_{a}) F(\omega,\sigma,\,.\,)(x) \,\frac{d\sigma}{\sigma} d\omega \\
		& \quad +  \int_{S^{d-1}} \int_0^1 \varphi_{\omega}(D_{a})\Psi(\sigma D_{a}) F(\omega,\sigma,\,.\,)(x) \,\frac{d\sigma}{\sigma} d\omega 
\end{align*}
for all  $F \in L^{2}(\R^{d}\times S^{d-1}\times(0,\infty);dx d\omega \frac{d\sigma}{\sigma}).$
\end{Def}

Note that $W_{a}$ is well defined thanks to Lemma \ref{cor:sim}, and that $\pi_a$ is the adjoint  of the operator $\bar W_a$, where  $\bar W_a$ is defined as $W_a$ with $D_a$ replaced by $D_a^\ast$.\\

\begin{Def}
Given $\omega \in S^{d-1}$, we fix vectors $\omega_{1},...,\omega _{d-1}$ such that
$\{\omega, \omega_{1},..., \omega_{d-1}\}$ is an orthonormal basis of $\R^{d}$. 
We then define the parabolic (quasi) distance in the direction of $\omega$ by
$$
d_{\omega}(x,y) := |\langle \omega, x-y \rangle| + \sum \limits _{j=1} ^{d-1}\langle \omega_{j}, x-y \rangle ^{2} \quad \forall x,y \in \R^{d}.
$$
We also define (anistropic) operators associated with this parabolic distance by
\begin{align*}
\Delta_{\omega^{\perp}} := \sum \limits _{j=1} ^{d-1} \langle \omega_{j}, \nabla\rangle^{2}, \quad
L_{\omega^{\perp}} := - \sum \limits _{j=1} ^{d-1} \langle \omega_{j}, D_{a}\rangle^{2}. 
\end{align*}
\end{Def}

\begin{Lemma}
\label{lem:kernelpsi} 
(i) Let $N \in \N$, $N>\frac{d+1}{2}$. There exists $C>0$ such that for all $\sigma \in (0,1)$ and $\omega \in S^{d-1}$, we have 
\begin{align*}
\| (1+\sigma  L_{\omega^{\perp}}  +\sigma^2 \langle \omega,D_a\rangle^2)^{-N}f\|_{L^{2}(\R^d)}  \leq C \sigma^{-\frac{d+1}{4}} \|f\|_{L^1(\R^d)}
\end{align*}
for all $f \in L^1(\R^d)$. \\
(ii) For every $M \in \N$, there exists $C_M>0$ such that for all $E,F \subset \R^d$ Borel sets, $\sigma \in (0,1)$ and $\omega \in S^{d-1}$, we have 
\begin{align*} 
\|1_E \psi_{\omega,\sigma}(D_a)(1_F f)\|_{L^{2}(\R^d)} 
	\leq C_M  \sigma^{-\frac{d}{2}} (1+ \frac{d_{\omega}(E,F)}{\sigma})^{-M} \|1_F f\|_{L^1(\R^d)}
\end{align*} 
for all $f \in L^1(\R^d)$. \\
 (iii) Let $1 \leq p \leq r < \infty$. For every $M \in \N$, there exists $C_M>0$ such that for all $E,F \subset \R^d$ Borel sets, $\sigma \in (0,1)$ and $\omega \in S^{d-1}$, we have 
\begin{align*} 
\|1_E \psi_{\omega,\sigma}(D_a)(1_F f)\|_{L^{r}(\R^d)} 
	\leq C_M  \sigma^{-d(\frac{1}{p}-\frac{1}{r})} \sigma^{-\frac{d-1}{4}}(1+ \frac{d(E,F)}{\sigma})^{-M} \|1_F f\|_{L^p(\R^d)}
\end{align*} 
for all $f \in L^p(\R^d)$. \\
\end{Lemma}

\begin{proof}
Part (i) follows from \cite[Proposition 4.3]{AMcT}, tracking the scaling factor $\sigma$ in its proof.

(ii)  Let $\omega \in S^{d-1}$.  For given Borel sets $E,F \subseteq \R^d$ with $d(E,F)>0$, let  $\chi_{\omega} \in C^\infty(\R^d)$ be a function with values in $[0,1]$ such that $\chi_{\omega} = \chi_{\omega}^{s}$, $\chi_{\omega}(\zeta)=0$ for $|\zeta| \leq \frac{1}{2} \kappa^{-1} d_{\omega}(E,F)$ and $\chi_{\omega}(\zeta)=1$ for $|\zeta| \geq \kappa^{-1} d_{\omega}(E,F)$, and $\|\langle \omega, \nabla \rangle \chi_{\omega}\|_{\infty} + \|\Delta_{\omega^{\perp}} \chi_{\omega}\|_{\infty} \lesssim \frac{1}{d_{\omega}(E,F)}$.  Lemma \ref{lem:od} implies
\begin{align*}
	c_d 1_E \psi_{\omega,\sigma}(D_a) 1_F f
	& = 1_E \int_{\R^d} \chi_{\omega}(\zeta) \calF^{-1}(\psi_{\omega,\sigma})(\zeta) e^{i\zeta D_a} 1_F f \,d\zeta.  
\end{align*}
Now note that 
$
	(1-\sigma \Delta_{\omega^{\perp}} -  \sigma^2 \langle \omega,\nabla_\zeta\rangle^2) e^{i\zeta D_a} 
	=(1+\sigma  L_{\omega^{\perp}}+\sigma^2 \langle \omega,D_a\rangle^2) e^{i\zeta D_a},
$
thus for $N \in \N$, 
\begin{align*}
	e^{i\zeta D_a} = (1+\sigma  L_{\omega^{\perp}}+\sigma^2 \langle \omega,D_a\rangle^2)^{-N} (1-\sigma \Delta_{\omega^{\perp}} -  \sigma^2 \langle \omega,\nabla_\zeta\rangle^2)^N e^{i\zeta D_a}. 
\end{align*}
From integration by parts we then get for  $j \in \{0,1\}$
\begin{align} \label{eq:osci} \nonumber 
	 c_d 1_E \psi_{\omega,\sigma}(D_a) 1_F f 
	  & = (1+\sigma  L_{\omega^{\perp}} +\sigma^2 \langle \omega,D_a\rangle^2)^{-N}  \\
	& \quad \quad \circ \int_{\R^d} ( (1-\sigma  \Delta_{\omega^{\perp}}  -  \sigma^2 \langle \omega,\nabla_\zeta\rangle^2)^N)^\ast (\chi_{\omega}^j \cdot \calF^{-1}(\psi_{\omega,\sigma}))(\zeta) e^{i\zeta D_a} (1_{F} f) \,d\zeta.
\end{align} 
Consider first the case   $d_{\omega}(E,F)\leq \sigma$,  for which we take $j=0$.  
According to Lemma \ref{lem:wavepackprop}, we have $\|\calF^{-1}(\psi_{\omega,\sigma})\|_{L^1(\R^d)} \lesssim \sigma^{-\frac{d-1}{4}}$. Similarly, one can check  that 
$$\|\zeta \mapsto (\sigma \langle \omega,\nabla_\zeta\rangle)^\beta  (\sigma\Delta_{\omega^{\perp}})^\alpha \calF^{-1}(\psi_{\omega,\sigma})(\zeta)\|_{L^1(\R^d)} \lesssim \sigma^{-\frac{d-1}{4}}$$
 for all $\alpha \in \N_0^d$ and $\beta \in \N_0$.  
We use this estimate together with Proposition \ref{prop:bddgrp} and Part (i) to obtain for $N>\frac{d+1}{2}$ 
\begin{align*}
	\|\psi_{\omega,\sigma}(D_a)f\|_{L^{2}(\R^d)} 
	\lesssim \sigma^{-\frac{d-1}{4}} \| (1+\sigma  L_{\omega^{\perp}} +\sigma^2 \langle \omega,D_a\rangle^2)^{-N}\|_{1 \to 2} \|f\|_{L^1(\R^d)} \lesssim \sigma^{-\frac{d}{2}} \|f\|_{L^1(\R^d)}. 
\end{align*}
In the case  $d_{\omega}(E,F)>\sigma$,  we choose $j=1$ in \eqref{eq:osci}. Then note that according to the choice of  $\chi_{\omega}$,  we have for $\sigma \in (0,1)$ that   $\|\zeta \mapsto (\sigma \langle \omega,\nabla_\zeta\rangle)^\beta  (\sigma\Delta_{\omega^{\perp}})^\alpha \chi_{\omega}(\zeta)\|_\infty \lesssim  (\frac{\sigma}{d_{\omega}(E,F)})^{|\alpha|+\beta} \lesssim 1$, for all $\alpha \in \N_0^d$, $\beta \in \N_0$.
Using the product rule, a version of \eqref{eq:sizeFourier-theta} for derivatives of $\calF^{-1}(\psi_{\omega,\sigma})$,  Part (i), and an anisotropic change of variable, we obtain 
\begin{align*}
	& \|1_E \psi_{\omega,\sigma}(D_a) (1_F f)\|_{2} \\
	& \quad \lesssim \sigma^{-\frac{d+1}{4}} \|1_F f\|_1  \sup_{\substack{\alpha \in \N_0^d, \, \beta \in \N_0 \\ |\alpha|+2\beta \leq N  }}  \int_{ \{ |\xi| \geq \frac{d(E,F)}{\kappa} \} \cap \{|\langle \omega,\xi\rangle| \geq \frac{\omega.d(E,F)}{\kappa} \}}  | (\sigma \langle \omega,\nabla_\zeta\rangle)^\beta (\sqrt{\sigma} \partial_\zeta)^\alpha\calF^{-1}(\psi_{\omega,\sigma})(\zeta)|\,d\zeta \\
	& \quad  \lesssim \sigma^{-\frac{d+1}{4}}  \sigma^{-\frac{3d+1}{4}}  \|1_F f\|_1  \int_{ \{ |\xi| \geq \frac{d(E,F)}{\kappa} \} \cap \{|\langle \omega,\xi\rangle| \geq \frac{\omega.d(E,F)}{\kappa} \}}  (1+\sigma^{-1}|\zeta|^2 + \sigma^{-2} \langle \omega,\zeta\rangle^2)^{-\tilde{N}} \,d\zeta \\
& \quad \lesssim \sigma^{-\frac{d}{2}} (1+ \frac{d_{\omega}(E,F)}{\sigma} )^{-(2\tilde{N}-d)}  \|1_F f\|_1. 
\end{align*} 
Choosing $\tilde{N}$ large enough in \eqref{eq:sizeFourier-theta} yields the result. \\
 (iii) This is similar to (i) and (ii), but simpler. By Theorem  \ref{thm:AMcT}, we have that 
\begin{align*}
\| (1+ \sigma^2 L)^{-N}f\|_{L^{r}(\R^d)}  \leq C \sigma^{-d(\frac{1}{p}-\frac{1}{r})} \|f\|_{L^p(\R^d)},
\end{align*}
for $N>\frac{d+1}{2}$. Integrating by parts, and using Lemma \ref{lem:od} together with Proposition \ref{prop:bddgrp}, we obtain that
\begin{align*} 
\|1_E \psi_{\omega,\sigma}(D_a)(1_F f)\|_{L^{r}(\R^d)} 
& \lesssim 
 \sigma^{-d(\frac{1}{p}-\frac{1}{r})} (1+\frac{d(E,F)}{\sigma})^{-M} 
 \int \limits _{\R^{d}} |(\sigma^{2} \Delta)^\alpha \calF^{-1}(\psi_{\omega,\sigma})| d\xi \cdot\|1_F f\|_{L^p(\R^d)}\\
 & \lesssim 
 \sigma^{-d(\frac{1}{p}-\frac{1}{r})}  \sigma^{-\frac{d-1}{4}}(1+\frac{d(E,F)}{\sigma})^{-M} \|1_F f\|_{L^p(\R^d)},
 \end{align*} 
using that, for all $\alpha \in \N$, 
$\|\zeta \mapsto (\sigma^{2} \Delta)^\alpha \calF^{-1}(\psi_{\omega,\sigma})(\zeta)\|_{L^1(\R^d)} \lesssim \sigma^{-\frac{d-1}{4}}$,
by Lemma \ref{lem:wavepackprop}.
\end{proof}

\section{The Hardy-Sobolev spaces $H^{p,s}_{FIO,a}(\R^d)$}
\label{sec:space}

In the following, we denote by  $\Psi \in C_c^\infty(\R^d)$  the function defining the wave packet transforms from Section \ref{sec:wavepackets}.
We denote by $H^1_L(\R^d)$ the Hardy space associated with $L$ as defined in \cite{DuongLi}. 
Recall that for all $f \in H^1_L(\R^d)$, we have by Lemma \ref{lem:sfequiv}, 
$$
		\|f\|_{H^1_L(\R^d)} \sim \|(\sigma,x) \mapsto \Psi(\sigma D_a) f(x)\|_{T^{1,2}(\R^d)}.
$$

\begin{Def}
Define
\begin{align*}
	\calS_1 =\{f \in H^1_L(\R^d)\,:\, \exists  g \in L^1(\R^{d})\cap L^{2}(\R^{d}) \;\; \exists \tau>0 \;\; f = \Psi(\tau D_{a})g\},
\end{align*}
and for $p \in (1,\infty)$
\begin{align*}
\mathcal{S}_{p} &= \{f \in L^p(\R^{d}) \, :\,  \exists g \in L^p(\R^{d})\cap L^{2}(\R^{d}) \;\; \exists \tau>0 \;\; f = \Psi(\tau D_{a})g\}.
\end{align*}
\end{Def}

\begin{Lemma}
\label{lem:welldef}
Let $p \in [1,\infty)$  and $f \in \mathcal{S}_{p}$. Then, for all $\omega \in S^{d-1}$, $\varphi_{\omega}(D_{a})f \in L^{p}(\R^{d})$, and, in the case $p=1$, $\varphi_{\omega}(D_{a})f \in H^1_L(\R^{d})$, each with norm independent of $\omega$. \end{Lemma}

\begin{proof}
We have that $\varphi_{\omega}(D_{a})f = \psi_{\omega,\tau}(D_{a})g$ for some $g \in L^{p}(\R^d)$,  up to a change of constants in the support conditions of $\psi_{\omega,\tau}$. 
By Lemma \ref{lem:kernelpsi}, we have $\psi_{\omega,\tau}(D_{a}) \in B(L^{p}(\R^{d}))$, and thus $\|\varphi_{\omega}(D_{a})f\|_{p} \lesssim_\tau \|g\|_{p}$.
In the case $p=1$, we obtain that $\|\psi_{\omega,\tau}(D_a) g\|_{L^{1}} \lesssim \|g\|_{H^{1}_L}$ by reasoning as in the proof of \ref{lem:kernelpsi} (iii), using the boundedness of Riesz transforms associated with $L$ from $H^{1}_{L}$ to $L^{1}$ to deduce the $H^{1}_{L}$ to $L^{1}$ uniform  boundedness of the transport group $(\exp(i\xi D_{a}))_{\xi \in \R^{d}}$.
We moreover have that $\psi_{\omega,\tau}(D_a) g \in R(L)$, since $\Psi$ is supported away from $0$, hence $\psi_{\omega,\tau}(D_a) g \in H^1_L(\R^d)$. 
\end{proof}

\begin{Cor}
\label{cor:welldef}
Let $p \in [1,\infty)$,  $s \in \R$,  and $f \in \mathcal{S}_{p}$. Then
$$
\omega \mapsto [(\sigma,x) \mapsto 1_{(1,\infty)}(\sigma)
\Psi(\sigma D_{a})f(x) + 1_{[0,1]}(\sigma) \sigma^{-s} 
\varphi_{\omega}(D_{a})\Psi(\sigma D_{a})f(x)]\in L^{p}(S^{d-1};T^{p,2}(\R^{d})).
$$
\end{Cor}

\begin{proof}
This follows from Lemma \ref{lem:welldef} and Theorem \ref{thm:HpD}.
\end{proof}

\begin{Lemma}
\label{lem:comp}
Let $\widetilde{\Psi} \in C_{c} ^{\infty} (\R^{d})$ be  non-degenerate, supported away from $0$
 and such that $\widetilde{\Psi}=\widetilde{\Psi}^{s}$.  Let $p \in (1,\infty)$, $s \in \R$,  and $f \in \mathcal{S}_{p}$. Then, we have that
 \begin{align*} 
\omega \mapsto [(\sigma,x) \mapsto 1_{(1,\infty)}(\sigma)
\widetilde{\Psi}(\sigma D_{a})f(x) + 1_{[0,1]}(\sigma)\sigma^{-s} 
\varphi_{\omega}(D_{a})\widetilde{\Psi}(\sigma D_{a})f(x)]\in L^{p}(S^{d-1};T^{p,2}(\R^{d})),
\end{align*}
with an equivalent norm to the corresponding map in Corollary \ref{cor:welldef},  and 
\begin{align*}
& \|(I+\sqrt{L})^{-M}f\|_{L^{p}} \\
	& \quad  \lesssim \|\omega \mapsto[(\sigma,x) \mapsto 1_{(1,\infty)}(\sigma)
\Psi(\sigma D_{a})f(x) + 1_{[0,1]}(\sigma) \sigma^{-s} 
\varphi_{\omega}(D_{a})\Psi(\sigma D_{a})f(x)]\|_{ L^p(S^{d-1};T^{p,2}(\R^{d}))},
\end{align*}
for all $M \in \N$ such that $M\geq \frac{d-1}{4}-s$. \end{Lemma}

\begin{proof}
Let $M \in \N$ be such that $M\geq \frac{d-1}{4}-s$.
Lemma \ref{lem:sfequiv} and Corollary \ref{cor:welldef} give the first part,  and Corollary  \ref{cor:funcalc}, Lemma \ref{lem:sfequiv} together with Theorem \ref{thm:HpD} give
\begin{align*}
\|(I+\sqrt{L})^{-M}f\|_{L^{p}} &  \lesssim \|(\sigma,x) \mapsto 1_{(1,\infty)}(\sigma)
\Psi(\sigma D_{a})(I+\sqrt{L})^{-M}f(x)\|_{T^{p,2}(\R^{d})} \\
&\qquad+ \|(\sigma,x)\mapsto 1_{[0,1]}(\sigma)
(\sigma \sqrt{L})^{M}(I+\sqrt{L})^{-M}\Psi^{ 2}(\sigma D_{a})f(x)\|_{T^{p,2}(\R^{d})}.\end{align*}
Using Corollary \ref{cor:funcalc} again, we then have that
\begin{align*}
\|(I+\sqrt{L})^{-M}f\|_{L^{p}} &  \lesssim \|(\sigma,x) \mapsto 1_{(1,\infty)}(\sigma)
\Psi(\sigma D_{a})f(x)\|_{T^{p,2}(\R^{d})} \\
&\qquad+ \|(\sigma,x)\mapsto 1_{[0,1]}(\sigma)
\sigma^{M}\Psi^{ 2}(\sigma D_{a})f(x)\|_{T^{p,2}(\R^{d})}.\end{align*}

We then use the reproducing formula \eqref{eq:repro} to obtain that 
\begin{align*}
& \|(I+\sqrt{L})^{-M}  f\|_{L^{p}} \\
& \qquad \lesssim \|(\sigma,x) \mapsto 1_{(1,\infty)}(\sigma)
\Psi(\sigma D_{a})f(x) + 1_{[0,1]}(\sigma)
\int \limits _{S^{d-1}} \sigma^{M-\frac{d-1}{4}} \varphi_{ \omega,\sigma}(D_{a})\Psi^{ 2}(\sigma D_{a})f(x) d\omega\|_{T^{p,2}(\R^{d})}\\
& \qquad \lesssim 
\|\omega \mapsto [(\sigma,x) \mapsto 1_{(1,\infty)}(\sigma)
\Psi(\sigma D_{a})f(x) + 1_{[0,1]}(\sigma) \sigma^{-s} 
\varphi_{\omega}(D_{a})\Psi(\sigma D_{a})f(x)]\|_{L^{p}(S^{d-1};T^{p,2}(\R^{d})},
\end{align*} 
since  $M\geq \frac{d-1}{4}-s$.
\end{proof}

\begin{Def}
\label{def:space}
Let $p \in [1,\infty)$,  and $s\in \R$. We define the space $H^{p,s}_{FIO,a}(\R^d)$  as the completion of $\mathcal{S}_{p}$ for the norm defined by
\begin{align*}
&\|f\|_{H^{p,s}_{FIO,a}(\R^{d})}  \\ 
& \quad :=  
\|\omega \mapsto [(\sigma,x) \mapsto 1_{(1,\infty)}(\sigma)
\Psi(\sigma D_{a})f(x) + 1_{[0,1]}(\sigma) \sigma^{-s} 
\varphi_{\omega}(D_{a})\Psi(\sigma D_{a})f(x)] \|_{L^{p}(S^{d-1};T^{p,2}(\R^{d}))}.
\end{align*}
We write $H^{p}_{FIO,a}(\R^d):= H^{p,0}_{FIO,a}(\R^d)$.  
\end{Def}

\begin{Remark}
By Lemma \ref{lem:comp}, we have that $H^p_{FIO,a}(\R^d)$ is a subspace of the $M$-th extrapolation space associated with $L$, and is independent of the choice of $\Psi \in C_{c} ^{\infty}(\R^{d}) \backslash \{0\}$, supported away from $0$,  and such that $\widetilde{\Psi}=\widetilde{\Psi}^{s}$. 
\end{Remark}

\begin{Remark}
By Lemma \ref{lem:repro}, interpolation properties of $H^{p,s}_{FIO,a}(\R^{d})$ follow from the interpolation properties of weighted tent spaces (see \cite{amenta}) with the same proof as in \cite[Proposition 6.7]{HPR}.
\end{Remark}

We also have the following versions of \cite[Theorem 4.1]{Rozendaal} and \cite[Corollary 4.4]{Rozendaal}, respectively.

\begin{Prop} \label{prop:vertical-conical}
Let $p \in (1,\infty)$, and $s \in \R$. Let $q \in C_c^\infty(\R^d)$  radial  with $q(\zeta) \equiv 1$ for $|\zeta| \leq \frac{1}{8}$. Then 
$$
		\|f\|_{H^{p,s}_{FIO,a}(\R^d)} \simeq \|q(D_{a})f\|_{L^p(\R^d)} + \left( \int_{S^{d-1}} \|\varphi_\omega(D_{a})(I+\sqrt{L})^{s}f\|_{L^p(\R^d)}^p\,d\omega \right)^{1/p} \quad \forall f \in \mathcal{S}_{p}.
$$
\end{Prop}

\begin{proof}
Let $f \in \mathcal{S}_{p}$. By Lemma \ref{lem:sfequiv}, we can choose $\Psi$ with an appropriate support, such that $\Psi(\sigma D_{a})f = \Psi(\sigma D_{a})q(D_{a})f$ for all $\sigma \geq 1$, $\Psi(\sigma D_{a})q(D_{a})=0$ for all $\sigma \leq \frac{1}{8}$, and $\varphi_{\omega}(D_{a})\Psi(\sigma D_{a}) =0$ for all $\sigma \geq 1$ and $\omega \in S^{d-1}$.

Then, by Theorem \ref{thm:HpD}, we have that
\begin{align*}
\|f\|_{H^{p,s}_{FIO,a}(\R^d)} &\lesssim \|(\sigma,x) \mapsto 1_{(1,\infty)}(\sigma)
\Psi(\sigma D_{a})q(D_{a})f(x)\|_{T^{p,2}(\R^{d})}\\ & \qquad\qquad  + \|\omega \mapsto [(\sigma,x)\mapsto 1_{[0,1]}(\sigma)\sigma^{-s} 
\varphi_{\omega}(D_{a})\Psi(\sigma D_{a})f(x)] \|_{L^{p}(S^{d-1};T^{p,2}(\R^{d}))} \\
&\lesssim \|q(D_{a})f\|_{L^p(\R^d)} + \left( \int_{S^{d-1}}  \|(I+\sqrt{L})^{s}\varphi_\omega(D_{a})f\|_{L^p(\R^d)}^p\,d\omega \right)^{1/p}. 
\end{align*}
In the other direction, Theorem \ref{thm:HpD} and the support properties of $q$ and $\Psi$ give us that
\begin{align*}
\|q(D_{a})f\|_{L^p(\R^d)} \lesssim \|f\|_{H^{p,s}_{FIO,a}(\R^d)} + \|(\sigma,x) \mapsto 1_{[\frac{1}{8},1]}(\sigma)
\Psi(\sigma D_{a})q(D_{a})f(x)\|_{T^{p,2}(\R^{d})}.
\end{align*}
With the same proof as in Lemma \ref{lem:sfequiv}, we then have that, for all $M \geq \frac{d-1}{4}-s$,
\begin{align*}
&\|(\sigma,x) \mapsto 1_{[\frac{1}{8},1]}(\sigma)
\Psi(\sigma D_{a})q(D_{a})f(x)\|_{T^{p,2}(\R^{d})} \\ & \quad \lesssim
\|(\sigma,x) \mapsto 1_{[\frac{1}{8},1]}(\sigma)
\int \limits _{0} ^{\infty}\Psi(\sigma D_{a})q(D_{a})\Psi(\tau D_{a})(I+\sqrt{L})^{M}(I+\sqrt{L})^{-M}
f(x)\frac{d\tau}{\tau}\|_{T^{p,2}(\R^{d})}\\ & \quad \lesssim \|(I+\sqrt{L})^{-M}f\|_{L^{p}(\R^{d})}.\end{align*}
Therefore, using Lemma \ref{lem:comp}, we have that $\|q(D_{a})f\|_{L^p(\R^d)} \lesssim \|f\|_{H^{p,s}_{FIO,a}(\R^d)}$.
For the second term, we use Theorem \ref{thm:HpD} and the support properties of $\Psi$ again to get that 
\begin{align*}
&\left( \int_{S^{d-1}} \|\varphi_\omega(D_{a})(I+\sqrt{L})^{s}f\|_{L^p(\R^d)}^p\,d\omega \right)^{1/p}
\\& \qquad \lesssim 
\|\omega \mapsto [ (\sigma,x) \mapsto 1_{[0,1)}(\sigma)\sigma^{-s}\varphi_\omega(D_{a})\Psi(\sigma D_{a})f(x)]\|_{L^{p}(S^{d-1};T^{p,2}(\R^{d}))}\\ &\qquad \lesssim \|f\|_{H^{p,s}_{FIO,a}(\R^d)}.
\end{align*}
\end{proof}

\begin{Prop} 
\label{prop:nontang}
Let $p \in (1,\infty)$.  Let $q \in C_c^\infty(\R^d)$  radial  with $q(\zeta) \equiv 1$ for $|\zeta| \leq \frac{1}{8}$,  and $\Phi \in \calS(\R^d)$ with $\Phi(0)=1$ and $\Phi_\sigma(\zeta) = \Phi(\sigma \zeta)$ for $\sigma>0$,  $\zeta \in \R^d$.  Then 
\begin{align*}
	 \|q(D_a) f\|_{L^p(\R^d)} 
	+ (\int_{S^{d-1}} \|(\sigma,\cdot)\mapsto  \Phi_\sigma(D_a) \varphi_\omega(D_a) f  \|_{T^{p,\infty}(\R^d)}^p \,d\omega)^{1/p} 
	\lesssim \|f\|_{H^p_{FIO,a}(\R^d)} 
	\quad \forall f \in \calS_p,
\end{align*}
and 
\begin{align*}
	 (\int_{S^{d-1}} \|  (\sigma,\cdot)\mapsto  \sigma^{\frac{d-1}{4}}  \Phi_\sigma(D_a) \varphi_\omega(D_a)^2 f  \|_{T^{p,\infty}(\R^d)}^p \,d\omega)^{1/p} 
	\lesssim \|f\|_{H^p_{FIO,a}(\R^d)} 
	\quad \forall f \in \calS_p.
\end{align*}
\end{Prop}

\begin{proof}
Let $r \in [1,p)$. 
For the first assertion, note that Theorem \ref{thm:AMcT} implies $L^r$-$L^\infty$ off-diagonal estimates for $\Phi_\sigma(D_a)$ of the following form: For every $M \in \N$, there exists $C_M>0$ such that for all $E,F \subset \R^d$ Borel sets, $\sigma \in (0,1)$, we have 
\begin{align*}
	\|1_E \Phi_\sigma(D_a) (1_F g) \|_{L^\infty(\R^d)} 
	\leq C_M \sigma^{-\frac{d}{r}} (1+\frac{d(E,F)}{\sigma})^{-M} \|1_F g\|_{L^r(\R^d)}
\end{align*}
for all $g \in L^r(\R^d)$. 
This implies that for $x \in \R^d$, 
\begin{align*}
		\sup_{|y-x| \leq \sigma} |\Phi_\sigma(D_a) g(y) |
	 \lesssim \sup_{|y-x| \leq \sigma} \sum_{j=0}^\infty 2^{-jM} (\sigma^{-d} \int_{S_j(B_{y,\sigma})} |g(z)|^r \,dz )^{1/r}
	 \lesssim M_rg(x),
\end{align*}
where $M_rg = (M(g^r))^{1/r}$, with $M$ the Hardy-Littlewood maximal function,
 $S_j(B_{y,\sigma}):= \{z \in \R^{d} \;;\; 2^{j-1}\sigma \leq |y-z| < 2^{j}\sigma\}$ for $j \geq 1$, and $S_{0}(B_{y,\sigma}) = \{z \in \R^{d} \;;\; |y-z| < \sigma\}$.  
The conclusion follows from the  $L^p(\R^d)$ boundedness of $M_r$ together with   Proposition \ref{prop:vertical-conical}. 

For the second assertion,  we first note that by renormalisation, we can change $\Phi_\sigma(D_a) \varphi_\omega(D_a)$ to $\Phi_\sigma(D_a)^2  \varphi_\omega(D_a)$.
We slightly change the above argument by noting that for  $q \in (r,\infty)$, we have $L^q$-$L^\infty$ off-diagonal estimates for $\Phi_\sigma(D_a)$. 
On the other hand, we have by Lemma \ref{lem:kernelpsi} $L^r$-$L^q$ off-diagonal estimates for $\Phi_\sigma(D_a) \varphi_\omega(D_a)$ of the form 
\begin{align*}
	\|1_E \Phi_\sigma(D_a) \varphi_\omega(D_a) (1_F g)\|_{L^q(\R^d)} 
	\leq C_M \sigma^{-d (\frac{1}{r}-\frac{1}{q})} \sigma^{-\frac{d-1}{4}} 
	(1+\frac{d(E,F)}{\sigma})^{-M} \|1_F g\|_{L^r(\R^d)}
\end{align*}
for all $g \in L^r(\R^d)$.  We then conclude as above,  using composition of off-diagonal bounds as in \cite[Theorem 2.3]{AM}. 
\end{proof} 

\section{Sobolev embedding properties of $H^{p}_{FIO,a}(\R^{d})$} 
\label{sec:sob}

We use a variation of the arguments in \cite[Section 7]{HPR}. \\We let  $m(D_a) = (I+\sqrt{L})^{-\frac{d-1}{4}}$. 

\begin{Lemma}
\label{lem:emb}
For every $0<\theta<\frac{\pi}{2}$  there exist  $C_\theta, c_{\theta}>0$  such that for all atoms $A \in T^{1,2}(\R^{d})$, and all $s \in \R$
\begin{equation} \label{eq:emb}
\int \limits _{S^{d-1}} \|(\sigma,x) \mapsto
1_{[0,1]}(\sigma) m(\sqrt{L})^{1+is}\psi_{\omega,\sigma}(D_{a}) A(\sigma,.)(x)\|_{T^{1,2}(\R^{d})} \,d\omega
\leq  C_\theta  e^{|s|c_{\theta}}. 
\end{equation}
\end{Lemma}

\begin{proof}
Let $A$ be a $T^{1,2}(\R^{d})$ atom associated with a ball $B=B(c_{B},r)$.
Without loss of generality, we assume that $A(\sigma,.)=0$ for all $\sigma \geq 1$. 

By renormalisation, we can replace $\psi_{\omega,\sigma}(D_a)$ in \eqref{eq:emb} by 
$\Psi_\sigma(D_a) \psi_{\omega,\sigma}(D_a)$.  Noting that $\|m^{is}\|_{L^\infty(S_\theta^o)} \leq ce^{|s|c_{\theta}}$,  for $c_{\theta} =\frac{\theta(d-1)}{4}$,  we  use Corollary \ref{cor:funcalc} to obtain for every $\omega \in S^{d-1}$ and given $\theta \in (0,\frac{\pi}{2})$
\begin{align*}
	& \|(\sigma,x) \mapsto
1_{[0,1]}(\sigma) m(D_{a})^{1+is}\Psi_\sigma(D_a)\psi_{\omega,\sigma}(D_{a}) A(\sigma,.)(x)\|_{T^{1,2}(\R^{d})} \\
 & \quad = \|(\sigma,x) \mapsto
1_{[0,1]}(\sigma) L^{\frac{d-1}{8}} m(D_{a})^{1+is}  \Psi_\sigma(D_a) L^{-\frac{d-1}{8}}  \psi_{\omega,\sigma}(D_{a}) A(\sigma,.)(x)\|_{T^{1,2}(\R^{d})} \\
& \quad \leq  C_\theta e^{|s|c_{\theta} }
 \|(\sigma,x) \mapsto
1_{[0,1]}(\sigma)  L^{-\frac{d-1}{8}}  \psi_{\omega,\sigma}(D_{a}) A(\sigma,.)(x)\|_{T^{1,2}(\R^{d})},
\end{align*}
with $C_\theta$ independent of $s \in \R$.
 
For $j \in \N^{*}$, and $\omega \in S^{d-1}$, define
$
C_{j,\omega}:= \{y \in \R^{d} \;;\; 2^{j-1}r < |\langle \omega, c_{B}-y \rangle| +
|c_{B}-y|^{2} \leq 2^{j}r\}
$
and $C_{0,\omega}:= \{y \in \R^{d} \;;\; |\langle \omega, c_{B}-y \rangle| +
|c_{B}-y|^{2} \leq r\}
 $.
Remark that $|C_{j,\omega}| \sim (2^{j}r)^{\frac{d+1}{2}}$,  and that
$d_{\omega}(C_{j,\omega},C_{0,\omega}) > 2^{j-1}r$. 
Using a slight generalisation of Lemma  \ref{cor:sim} and Corollary \ref{cor:sobemb} for $p=\frac{4d}{3d-1}$, we have that
 
\begin{align*}
(\int \limits _{S^{d-1}} & \|(\sigma,x) \mapsto 1_{C_{0,\omega}}(x)
1_{[0,1]}(\sigma) L^{-\frac{d-1}{8}} \psi_{\omega,\sigma}(D_{a}) A(\sigma,.)(x)\|_{T^{1,2}(\R^{d})} d\omega )^{2}\\
& \lesssim r^{\frac{d+1}{2}}\int \limits _{S^{d-1}}  
\int \limits _{0} ^{\min(r,1)} 
\|L^{-\frac{d-1}{8}}\psi_{\omega,\sigma}(D_{a})  A(\sigma,.)(x)\|_{L^{2}(\R^{d})} ^{2} \frac{d\sigma}{\sigma} d\omega \\
& \lesssim r^{\frac{d+1}{2}}
\int \limits _{0} ^{\min(r,1)} 
 \|L^{-\frac{d-1}{8}}A(\sigma,.)(x)\|_{L^{2}(\R^{d})} ^{2} \frac{d\sigma}{\sigma}\\
& \lesssim r^{\frac{d+1}{2}}
\int \limits _{0} ^{r} 
 \|A(\sigma,.)(x)\|_{L^{p}(\R^{d})} ^{2} \frac{d\sigma}{\sigma}  \\
& \lesssim r^{\frac{d+1}{2}} r^{\frac{d-1}{2}}
 \int \limits _{0} ^{r} 
 \|A(\sigma,.)(x)\|_{L^{2}(\R^{d})} ^{2} \frac{d\sigma}{\sigma} 
 \lesssim r^{d}\|A\|_{T^{2,2}} ^{2} \lesssim 1.
\end{align*}

Let $M > d+1$, and define $\widetilde{\Psi}: \xi \mapsto \frac{|\xi|^{-\frac{d-1}{4}} \Psi(\xi)}{
(\int \limits _{0} ^{\infty} |\sigma\xi|^{-\frac{d-1}{2}} |\Psi(\sigma \xi)|^{2} \frac{d\sigma}{\sigma})^{\frac{1}{2}}}$, and
$\tilde{\psi}_{\omega, \sigma}: \xi \mapsto \varphi_{\omega, \sigma}(\xi) \widetilde{\Psi}(\sigma \xi)$. 

For all $j \in \N^{*}$,  we obtain from Lemma \ref{lem:kernelpsi} for $\widetilde{\psi_{\omega,\sigma}}$ instead of $\psi_{\omega,\sigma}$ 
\begin{align*}
(\int \limits _{S^{d-1}} & \|(\sigma,x) \mapsto 1_{C_{j,\omega}}(x)
1_{[0,1]}(\sigma) L^{-\frac{d-1}{8}} \psi_{\omega,\sigma}(D_{a}) A(\sigma,.)(x)\|_{T^{1,2}(\R^{d})} d\omega )^{2}\\
& \lesssim (2^{j}r)^{\frac{d+1}{2}}\int \limits _{S^{d-1}}  
\int \limits _{0} ^{\min(r,1)} \sigma^{\frac{d-1}{2}}
\| \widetilde{\psi_{\omega, \sigma}}(D_{a}) A(\sigma,.)\|_{L^2(C_{j,\omega})}^{2} \, \frac{d\sigma}{\sigma} d\omega\\ 
& \lesssim  (2^{j}r)^{\frac{d+1}{2}}\int \limits _{S^{d-1}}  
\int \limits _{0} ^{\min(r,1)} \sigma^{\frac{d-1}{2}} \sigma^{-d} \left(\frac{\sigma}{2^jr}\right)^M \|A(\sigma,.)\|_{L^1(\R^d)}^2 \, \frac{d\sigma}{\sigma} d\omega \\
& \lesssim r^d  \int \limits _{S^{d-1}}  
\int \limits _{0} ^{\min(r,1)} (\frac{2^{j}r}{\sigma})^{\frac{d+1}{2}} \left(\frac{\sigma}{2^jr}\right)^M \|A(\sigma,.)\|_{L^2(\R^d)}^2 \, \frac{d\sigma}{\sigma} d\omega \\
& \lesssim 2^{-j(M-\frac{d+1}{2})}  r^d \|A\|_{T^{2,2}}^2 \lesssim    2^{-j(M-\frac{d+1}{2})} .
\end{align*}
Summing over $j$ yields the conclusion. 
\end{proof}

\begin{Remark}
\label{rem:tentbdd}
Note that basically the same proof as above also yields the statement  that for all $s \in \R$,
\begin{align*}
	\|(\omega,\sigma,\,.\,) \mapsto   \sigma^{\frac{s_1}{2}+is} \psi_{\omega,\sigma}(D_a)F(\sigma,\,.\,)\|_{L^1(S^{d-1};T^{1,2}(\R^d))} \lesssim \|F\|_{T^{1,2}(\R^d)}
\end{align*}
for all $F \in T^{1,2}(\R^d)$. 
By  a slight modification of Lemma \ref{cor:sim}, we obtain on the other hand 
$
	\|(\omega,\sigma,\,.\,) \mapsto    \psi_{\omega,\sigma}(D_a)F(\sigma,\,.\,)\|_{L^2(S^{d-1};T^{2,2}(\R^d))} \lesssim \|F\|_{T^{2,2}(\R^d)}
$
for all $F \in T^{2,2}(\R^d)$.  Stein interpolation and duality then yield for all $p \in (1,\infty)$, 
\begin{align*}
	\|(\omega,\sigma,\,.\,) \mapsto   \sigma^{\frac{s_p}{2}} \psi_{\omega,\sigma}(D_a)F(\sigma,\,.\,)\|_{L^p(S^{d-1};T^{p,2}(\R^d))} \lesssim \|F\|_{T^{p,2}(\R^d)},
\end{align*}
for all $F \in T^{p,2}(\R^d)$. 
\end{Remark}

\begin{Lemma}
\label{lem:emb2}
For all $p \in [1,2]$, and $s_{p} = (d-1)(\frac{1}{p}-\frac{1}{2})$, we have the continuous inclusion
$H^{p,\frac{s_{p}}{2}}_{FIO,a}(\R^{d}) \subset H^{p} _{L}(\R^d)$, where $H^{p} _{L}(\R^d)=L^{p}(\R^d)$ for $p>1$. 
For $p \in (1,\infty)$, and $b: \xi \mapsto |\xi|^{\frac{d-1}{4}} m(\xi)$, we have that
\begin{align*}
&\|(\sigma,x)\mapsto m(D_{a}) \Psi(\sigma D_{a})f(x)\|_{T^{p,2}(\R^{d})} \lesssim 
\|(b(D_{a})+m(D_{a}))f\|_{H^{p}_{FIO,a}(\R^{d})} \lesssim \|f\|_{H^{p}_{FIO,a}(\R^{d})},  
\end{align*}
for all $f \in \mathcal{S}_{p}$.
\end{Lemma}

\begin{proof}
Let $f$ be an $H^{1}_L$ atom.
We have,  
using the reproducing formula \eqref{eq:repro},  that
\begin{align*}
\|f\|_{H^{1}_L} 
&\sim 
\|(\sigma,x)\mapsto \Psi(\sigma D_{a})f(x)\|_{T^{1,2}(\R^{d})}
\\ & \lesssim \int \limits _{S^{d-1}} \|(\sigma,x)\mapsto 1_{[0,1]}(\sigma) \sigma^{-\frac{d-1}{4}} \psi_{\omega,\sigma}(D_{a})f(x)+1_{[1,\infty)}(\sigma) \Psi(\sigma D_{a})f(x)\|_{T^{1,2}(\R^{d})} d\omega 
\\
& \lesssim  
\|f\|_{H^{1,\frac{d-1}{4}}_{FIO,a}(\R^{d})},
\end{align*}
where the last inequality follows from the comparability of $\psi_{\omega,\sigma}$ with $\varphi_\omega\Psi_\sigma$ for $\sigma \in (0,1)$. 
Since $H^{2}_{FIO,a}=L^{2}$, the  continuous inclusion
$H^{p,\frac{s_{p}}{2}}_{FIO,a}(\R^{d}) \subset H^{p} _{L}(\R^d)
$ follows by interpolation.
In the same way,
\begin{align*}&\|(\sigma,x)\mapsto 1_{[0,1]}(\sigma) m(D_{a})\Psi(\sigma D_{a})f(x)\|_{T^{p,2}(\R^{d})}
\\ &\quad \lesssim 
\int \limits _{S^{d-1}} \|(\sigma,x)\mapsto 1_{[0,1]}(\sigma) b(D_{a})\varphi_{\omega}(D_{a})\widetilde{\Psi}(\sigma D_{a})f(x)\|_{T^{p,2}(\R^{d})} d\omega,
\end{align*}
for $\widetilde{\Psi}$ such that $\Psi(\xi)= |\xi|^{\frac{d-1}{4}}\widetilde{\Psi}(\xi)$ for all $\xi \in \R^{d}$. Turning to the low frequency term, we note that, for $\sigma>1$, we have that $\Psi(\sigma \xi) = \Psi(\sigma \xi)q(\xi)$ for all $\xi \in \R^{d}$. Therefore, by Theorem \ref{thm:HpD} and Proposition \ref{prop:vertical-conical} we have that
\begin{align*}
\|(\sigma,x) \mapsto 1_{(1,\infty)}(\sigma) 
\Psi(\sigma D_{a})m(D_{a})f(x)\|_{T^{p,2}(\R^{d})}
&\lesssim \|m(D_{a})q(D_{a})f\|_{L^{p}(\R^{d})} \lesssim \|m(D_{a})f\|_{H^{p}_{FIO,a}(\R^{d})}.
\end{align*}
To conclude the proof, we use Theorem \ref{thm:mult} and Theorem \ref{thm:transf}, along with Proposition \ref{prop:bddgrp}, to show that $b(D_{a})$ and $m(D_{a})$ are bounded operators on $L^{p}(\R^{d})$, and thus also on $H^{p}_{FIO,a}(\R^{d})$, thanks to Proposition \ref{prop:vertical-conical}.
\end{proof}

\begin{Cor}
\label{cor:sob}
Let $p \in (1,2]$. Then 
\begin{align*}
&\|(I+\sqrt{L})^{-\frac{s_{p}}{2}} f\|_{H^{p}_{FIO,a}(\R^{d})}  \lesssim \|f\|_{L^{p}(\R^{d})},
\end{align*}
for all $f \in \mathcal{S}_{p}$.
\end{Cor}

\begin{proof}
For $z \in \C$ such that $Re(z) \in [0,1]$, we consider the operators defined by
$$
T_{z}f(x,\omega,\sigma):= 1_{[0,1]}(\sigma) (I+\sqrt{L})^{-(\frac{d-1}{4})z}\psi_{\omega,\sigma}(D_{a})f(x) \quad \forall f \in L^{2}(\R^{d}).
$$
For $Re(z) = 0$, they are well defined as operators from $L^{2}(\R^{d})$ to $L^{2}(\R^{d} \times S^{d-1} \times (0,\infty);dx d\omega \frac{d\sigma}{\sigma})$ by Lemma \ref{cor:sim}, with  norm independent of $Im(z)$.  
For $Re(z) = 1$, by Lemma \ref{lem:emb}, $T_{z}$ extends to a bounded operator from $H^{1}(\R^{d})$ to $L^{1}(S^{d-1};T^{1,2}(\R^{d}))$ with norm bounded by $C_\theta e^{|Im(z)|c_{\theta}}$ for fixed $\theta>0$. 
Therefore, by Stein interpolation \cite{Stein} with admissible growth, $T_{z} \in B(L^{p}(\R^{d}),L^{p}(S^{d-1};T^{p,2}(\R^{d}))$
for $Re(z)= \frac{2}{p}-1$. To conclude the proof, we thus only have to show the low frequency estimate
$$
\|(\sigma,x) \mapsto 1_{(1,\infty)}(\sigma) \Psi(\sigma D_{a})(I+\sqrt{L})^{-\frac{s_{p}}{2}}f(x)\|_{T^{p,2}(\R^{d})} 
\lesssim \|f\|_{L^{p}(\R^{d})}.
$$
This follows from Theorem \ref{thm:HpD} and the $L^p$ boundedness of $(I+\sqrt{L})^{-\frac{s_{p}}{2}}$.
\end{proof}

\section{The wave group}
\label{sec:wave}

\begin{Theorem} \label{thm:main}
Let $p \in (1,\infty)$, and $s \in \R$.   Then 
$$
		e^{it\sqrt{L}}: H^{p,s}_{FIO,a}(\R^d) \to H^{p,s}_{FIO,a}(\R^d) 
$$
is bounded for each $t>0$.
\end{Theorem}

For simplicity, we set $t=1$ and $s=0$. All the proofs extend verbatim to other values of $t$. The case $s \in \R$ is an immediate consequence of the case $s=0$ by Proposition \ref{prop:vertical-conical}. 
For the transport groups,  and the one dimensional wave groups,  the $L^p$ boundedness is clear. 

\begin{Lemma} \label{lem:transport}
Let $p \in (1,\infty)$ and $\omega \in S^{d-1}$. Then  $$e^{i\omega.\sqrt{D_{a}^{2}}} \in B(L^p(\R^d)) \cap B(H^p_{FIO,a}(\R^d)).$$ 
\end{Lemma}

\begin{proof}
The $L^p$ boundedness is proven in Proposition \ref{prop:bddgrp}. The boundedness on $H^p_{FIO,a}(\R^d)$ is  an immediate consequence of the $L^p$ boundedness, by Proposition \ref{prop:vertical-conical}.
\end{proof}

For the low frequency estimate, we need the following lemma. 
\begin{Lemma} \label{lem:low-freq} 
Let $p \in (1,\infty)$, let $q \in C^\infty_c(\R^d)$  be radial.  Then 	$q(D_a)e^{i\sqrt{L}}: L^p(\R^d) \to L^p(\R^d)$ is bounded. 
\end{Lemma} 

\begin{proof}
Because of the compact support of $q$, the symbol $m:\zeta \mapsto q(\zeta)e^{i|\zeta|}$ clearly satisfies the Marcinkiewicz-Lizorkin multiplier condition of Theorem \ref{thm:mult}. The result thus follows from Theorem \ref{thm:mult} and Theorem \ref{thm:transf} using that  $(e_{j}\sqrt{D_a^{2}})_{j=1,...,d}$ generates a bounded commutative  $d$-parameter group (as shown in Proposition \ref{prop:bddgrp}),  along with the fact that
$$
m(D_{a}) = m^{s}(D_{a}) = \frac{1}{(2\pi)^{d}} \int \limits _{\R^{d}} \widehat{m}(\xi)\exp(i\xi\sqrt{D_{a}^{2}})d\xi,
$$
as explained in Definition \ref{def:calc}. 
\end{proof}

\begin{proof}[Proof of Theorem \ref{thm:main}]
 For $f \in \mathcal{S}_{p}$, Proposition \ref{prop:vertical-conical} yields 
\begin{align*}	
	\|e^{i\sqrt{L}}f\|_{H^p_{FIO,a}(\R^d)} 
	& \lesssim \|q(D_a)e^{i\sqrt{L}}f\|_{L^p(\R^d)} + \left(\int_{S^{d-1}} \|\varphi_\omega(D_a) e^{i\sqrt{L}} f\|_{L^p(\R^d)}^p\,d\omega\right)^{1/p}.
\end{align*}

For the low frequency part, recall that $q \in C^\infty_c(\R^d)$ with $q(\zeta) \equiv 1$ for  $|\zeta|\leq \frac{1}{8}$.  Choose  $\tilde q \in C^\infty_c(\R^d)$  radial  with $\tilde q(\zeta)\equiv 1$ on $\supp q$. Then $q(D_a)e^{i\sqrt{L}} = \tilde q(D_a)e^{i\sqrt{L}} q(D_a)$, since  $\sqrt{D_a^{2}}$  and $\sqrt{L}$ are commuting, and $\tilde q(D_a)e^{i\sqrt{L}}$ is $L^p$ bounded according to Lemma \ref{lem:low-freq}. Thus, 
\begin{align*}
	\|q(D_a)e^{i\sqrt{L}} f\|_{L^p(\R^d)} 
	= \|\tilde q(D_a) e^{i\sqrt{L}} q(D_a) f\|_{L^p(\R^d)}  
	\lesssim \|q(D_a) f\|_{L^p(\R^d)} .
\end{align*}

Let us now consider the high frequency part. 
For fixed $\omega \in S^{d-1}$, we decompose 
\begin{align*}
	\varphi_\omega(D_a)e^{i\sqrt{L}} = \varphi_\omega(D_a)  e^{i\omega.\sqrt{D_a^{2}}}  + \varphi_\omega(D_a)(e^{i\sqrt{L}} -  e^{i\omega.\sqrt{D_a^{2}}}).
\end{align*}
The first part can be dealt with Lemma \ref{lem:transport}, which directly yields
$$
		 \left(\int_{S^{d-1}} \|\varphi_\omega(D_a)  e^{i\omega.\sqrt{D_a^{2}}}  f\|_{L^p(\R^d)}^p\,d\omega\right)^{1/p}
		\lesssim \|f\|_{H^p_{FIO,a}(\R^d)}.
$$

For the second part, we use \eqref{eq:resol-identity}  to write 
\begin{align*}
	\varphi_\omega(D_a)(e^{i\sqrt{L}} - e^{i\omega.\sqrt{D_a^{2}}} ) 
		= \varphi_\omega(D_a)  e^{i\omega.\sqrt{D_a^{2}}} (e^{-i\omega.\sqrt{D_a^{2}}}  e^{i\sqrt{L}} -I) \pi_a W_a.
\end{align*}
Since  $e^{i\omega.\sqrt{D_a^{2}}}$ is bounded on $L^p(\R^d)$ by Lemma \ref{lem:transport}, it suffices to show that 
\begin{align*}
	\|\varphi_\omega(D_a) ( e^{-i\omega.\sqrt{D_a^{2}}}  e^{i\sqrt{L}} -I) \pi_a W_a f\|_{L^p(\R^d)} 
	\lesssim \|\varphi_\omega(D_a)f\|_{L^p(\R^d)}. 
\end{align*}
We can write 
\begin{align*}
	\varphi_\omega(D_a) ( e^{-i\omega.\sqrt{D_a^{2}}}  e^{i\sqrt{L}} -I) \pi_a W_a 
	=m_\omega(D_a)\varphi_\omega(D_a) + q_\omega(D_a) \varphi_\omega(D_a)
\end{align*}
for the symbols 
\begin{align} \label{eq:def-multiplier}
m_\omega(\zeta)
	 = \tilde \varphi_\omega(\zeta) \tilde m_\omega(\zeta) \int_0^1 \int_{S^{d-1}} \psi_{\nu,\sigma}(\zeta)^2\,d\nu \frac{d\sigma}{\sigma} 
\end{align}
and 
\begin{align*}
	q_\omega(\zeta) 
	=\tilde \varphi_\omega(\zeta) \tilde m_\omega(\zeta)   r(\zeta)^2 
\end{align*}
with $\tilde m_\omega(\zeta)=e^{- i\sum _{j=1} ^{d} \omega_{j}|\zeta_{j}|  +i|\zeta|}-1$,  
 $\tilde\varphi_\omega \in C_c^\infty(\R^d)$ a function with $\tilde \varphi_\omega \equiv 1$ on $\supp \varphi_\omega$ and $\tilde \varphi_\omega(\zeta)=0$ for $|\zeta|<\frac{1}{16}$ or $ \underset{(\varepsilon_{j})_{j=1} ^{d} \in \{-1,1\}^{d}}{\min}
	|(\varepsilon_{1}\hat \zeta_{1},...,\varepsilon_{d}\hat \zeta_{d}) - \omega|  >4|\zeta|^{-1/2}$, and 
$$
		r(\zeta):=\left( \int_1^\infty \Psi_\sigma(\zeta)^2 \,\frac{d\sigma}{\sigma} \right)^{1/2}, \quad \zeta \neq 0,
$$
and $r(0):=1$. As noted  in \cite[Section 4.1]{HPR}, we have $r \in C_c^\infty(\R^d)$. 
 
The proof will be concluded by applying Theorem \ref{thm:mult}, and Theorem \ref{thm:transf}, using Proposition \ref{prop:bddgrp}. We only have to check that $m_\omega$ and $q_\omega$ 
satisfy the assumption of Theorem \ref{thm:mult}. For $q_\omega$, this directly follows from the fact that $r \in C_c^\infty(\R^d)$. For $m_\omega$, this is proven in Lemma \ref{lem:Marcinkiewicz} below.
\end{proof}

\begin{Remark} \label{rem:cut-off-cond}
Let $\omega \in S^{d-1}$. 
Let $\tilde \varphi_\omega \in C_c^\infty(\R^d)$ a function with $\tilde \varphi_\omega \equiv 1$ on $\supp \varphi_\omega$ and $\tilde \varphi_\omega(\zeta)=0$ for $|\zeta|<\frac{1}{16}$ or $ \underset{(\varepsilon_{j})_{j=1} ^{d} \in \{-1,1\}^{d}}{\min}
	|(\varepsilon_{1}\hat \zeta_{1},...,\varepsilon_{d}\hat \zeta_{d}) - \omega|  >4|\zeta|^{-1/2}$. 
By the choice of the cut-off function $\tilde \varphi_\omega$ and the support properties of $\varphi_\omega$, we have the following: 
For all $\alpha \in \N_0^d$ and $\beta \in \N_0$, there exists a constant $C=C(\alpha,\beta)>0$ such that 
$$
		|\langle \omega,\nabla_\zeta\rangle^\beta \partial_\zeta^\alpha \tilde \varphi_\omega(\zeta)|
		\leq C |\zeta|^{-\frac{|\alpha|}{2}-\beta}
$$
for all $\omega \in S^{d-1}$ and $\zeta \in \R^d \setminus \{0\}$.
\end{Remark}


\begin{Lemma} \label{lem:Marcinkiewicz}
Let $\omega \in S^{d-1}$, let $m_\omega$ be as defined in  \eqref{eq:def-multiplier}. 
For all $\alpha \in \N_0^d$ with $|\alpha|_\infty\leq 1$ there exists a constant $C=C(\alpha)>0$ such that 
$$
		|\zeta^\alpha \partial_\zeta^\alpha m_{\omega}(\zeta)| \leq C 
$$
for all $\zeta \in \R^d \setminus \{0\}$. 
\end{Lemma}

\begin{proof}
By rotational invariance it suffices to consider the case $\omega=e_1$.
Let $\zeta \in \R^d \setminus \{0\}$.
The bound $|m_{e_1}(\zeta)| \leq C$ directly follows from \eqref{eq:psi-identity} and the boundedness of $\tilde m_{e_1}$ and $\tilde \varphi_{e_1}$. 
Moreover, by the specific form of $\tilde m_{e_1}(\zeta) = e^{ib(\zeta)}-1$ with $b(\zeta)= -|\zeta_1| +|\zeta|$, it can easily be seen that the condition 
\begin{equation} \label{eq:mult-main-cond}
	|\zeta^\alpha \partial_\zeta^\alpha b(\zeta)| \leq c
\end{equation} 
 for $|\alpha|_\infty \leq 1$ immediately implies $|\zeta^\alpha \partial_\zeta^\alpha \tilde m_{e_1}(\zeta)| \leq c$  for $|\alpha|_\infty \leq 1$. We check \eqref{eq:mult-main-cond}: 
\begin{align*}
	|\zeta_1 \partial_1 b(\zeta)| = |\zeta_1 \partial_1 (- |\zeta_1|  + |\zeta|)|
	& \leq |\zeta_1||1-  \frac{|\zeta_1|}{|\zeta|} |
	= \left|\frac{\zeta_1}{|\zeta|}\right|||\zeta|- |\zeta_1| | \\
	&\leq ||\zeta|-|\zeta_1| |
	=|\zeta_1|\left(\sqrt{1+\sum_{j=2}^d \frac{\zeta_j^2}{\zeta_1^2}} -1 \right).
\end{align*}
According to the support properties of $\tilde \varphi_{e_1}$ and $\psi_{\nu,\sigma}$, we have $|\nu- \varepsilon_{1} e_1|\lesssim \sqrt{\sigma}$  for some $\varepsilon_{1} \in \{-1,1\}$.  Thus a slight modification of \eqref{eq:zetacomp} yields that there exist constants $c_1,c_2>0$ such that for $0<\sigma \ll 1$, one has 
\begin{equation} \label{eq:cond-zeta}
	 |\zeta_1| >\frac{c_1}{\sigma} \qquad \text{and} \qquad |\zeta_j|\leq \frac{c_2}{\sqrt{\sigma}}, \quad  j \in \{2,\ldots,d\},
\end{equation}
 on the support of $m_{e_1}$.
Thus, for such choice of $\zeta$, \begin{align*}
|\zeta_1 \partial_1 b(\zeta)| \lesssim	|\zeta_1|\left(\sqrt{1+\frac{c}{ |\zeta_1|} } -1 \right).
\end{align*}
This expression remains bounded for $ |\zeta_1|  \to \infty$ or equivalently $|\zeta|\to \infty$, since replacing $h=\frac{1}{ |\zeta_1| }$, we see that 
\begin{align*}
	\lim_{h\to 0} \frac{\sqrt{1+ch}-1}{h} =\frac{c}{2}.
\end{align*}
Again using \eqref{eq:cond-zeta} and $|\zeta|\geq |\zeta_1| >\frac{c_1}{\sigma}$, we obtain for $j \in \{2,\ldots,d\}$ that
\begin{align*}
	|\zeta_j \partial_j b(\zeta)| =|\zeta_j \partial_j  (- |\zeta_1|  +  |\zeta|)| \leq |\zeta_j \frac{\zeta_j}{|\zeta|}| \leq c.
\end{align*}
Concerning the mixed derivatives, one can  inductively show that for $\alpha \in \N_0^d$ with $|\alpha|_\infty \leq 1$ and $\alpha_1=0$, 
$
		|\zeta^\alpha \partial_\zeta^\alpha b(\zeta)| = |\frac{\zeta^{2\alpha}}{|\zeta|^{2|\alpha|-1}}|\leq c,
$
for $\zeta$ as in \eqref{eq:cond-zeta}. Finally, for $j \neq 1$, 
\begin{align*}	
	|\zeta_1 \zeta_j \partial_1 \partial_j b(\zeta)| 
	&=	|\zeta_1 \zeta_j \partial_1 \partial_j(- |\zeta_1|  +|\zeta|)|
		=|\zeta_1 \zeta_j| |\frac{\zeta_1\zeta_j}{|\zeta|^3}| \leq c. 
\end{align*}
Putting all arguments together shows \eqref{eq:mult-main-cond}. 
The  bound $|\zeta^\alpha \partial_\zeta^\alpha \tilde \varphi_{e_1}(\zeta)| \leq c$ follows from Remark \ref{rem:cut-off-cond} together with \eqref{eq:cond-zeta}, whereas the analogous bound for the last factor in \eqref{eq:def-multiplier} concerning $\psi_{\nu,\sigma}$ is a consequence of \eqref{eq:cond-psi} together with \eqref{eq:cond-zeta}. 
\end{proof}

Combining Corollary \ref{cor:sob} with Theorem \ref{thm:main} and Theorem \ref{thm:HpD} then gives our main result.

\begin{Theorem}
\label{thm:main1}
Let $p \in (1,\infty)$ and $s_{p} = (d-1)|\frac{1}{p}-\frac{1}{2}|$. 
For each $t\in \R$, the operator $(I+\sqrt{L})^{-s_{p}}\exp(it\sqrt{L})$ is bounded on $L^{p}(\R^{d})$. Moreover, if $s_{p} \leq 2$, the operator $exp(it\sqrt{L})$ is bounded from 
$W^{s_{p},p}(\R^d)$ to $L^{p}(\R^d)$.
\end{Theorem}

\begin{proof}
By duality, it suffices to consider the case $p \in (1,2)$.
Let $f \in \mathcal{S}_{p}$. 
By Lemma \ref{lem:emb2} and Theorem \ref{thm:main}, we have that 
\begin{align*}
\|\exp(it\sqrt{L})f\|_{L^{p}(\R^{d})} &\lesssim
\|\exp(it\sqrt{L})f\|_{H^{p,\frac{s_{p}}{2}}_{FIO,a}(\R^{d})}
\lesssim \|f\|_{H^{p,\frac{s_{p}}{2}}_{FIO,a}(\R^{d})}.
\end{align*}
Using Proposition \ref{prop:vertical-conical}, and Corollary \ref{cor:sob}, we then have that
\begin{align*}
&\|\exp(it\sqrt{L})f\|_{L^{p}(\R^{d})} \lesssim \|(I+\sqrt{L})^{\frac{s_{p}}{2}}f\|_{H^{p}_{FIO,a}(\R^{d})} \lesssim \|(I+\sqrt{L})^{s_{p}}f\|_{L^{p}(\R^{d})}.
\end{align*}
For $s_{p} \leq 2$, Theorem \ref{thm:HpD} then gives
$\|f\|_{W^{s_{p},p}} \sim \|(I+\sqrt{L})^{s_{p}}f\|_{L^{p}(\R^{d})}$.
\end{proof}

\section{Lower order perturbations} 
\label{sec:perturb}

We consider the operators $L_{1}:=- \sum \limits _{j=1} ^{d} 
\widetilde{a_{j+d}}\partial_{j}\widetilde{a_{j}}\partial{j}$ and $L_{2}:=- \sum \limits _{j=1} ^{d} 
\widetilde{a_{j}}\partial_{j}\widetilde{a_{j+d}}\partial{j}$.
For a function $g: \R^{d} \to \R$, we denote by $M_g$ the multiplication operator $(f,F) \mapsto (gf,gF)$. 
We will evaluate the norm of $g$ in Besov spaces 
$\dot{B}_{\infty,\infty} ^{0,L_{k}}$ associated with the operators $L_{k}$, in the sense of \cite{bdy}, 
as well as in $BMO_{L_{k}}$ spaces, in the sense of \cite{dy}.


\begin{Theorem} 
\label{thm:perturb}
Let $p \in (1,\infty)$ and $s_p=(d-1)|\frac{1}{p}-\frac{1}{2}|$. 
Let $g \in L^{\infty}$  be such that $g \in \dot{B}_{\infty,\infty} ^{0,L_{m}}$, 
$\nabla L_{m} ^{-\frac{1}{2}} g \in \dot{B}_{\infty,\infty} ^{0,L_{m}}$ and $L_{m}^{s_{p}}g \in BMO_{L_{m}}$
for $m=1,2$.
Then $M_g \in B(H^{p}_{FIO,a}(\R^d))$.  
\end{Theorem}

\begin{proof}
For $p =2$, there is nothing to prove. For $p \neq 2$, this is a consequence of Lemma \ref{lem:para1} and Lemma  \ref{lem:para2} below. 
\end{proof}

\begin{Remark}
 If the coefficients $(a_{j})_{j=1,...,2d}$ are $C^{1,\alpha}$ for some $\alpha \in (0,1]$, then  \cite[Theorem 4.19]{AMcT} implies that 
$$
\underset{m=1,2}{\max} \|g\|_{\dot{B}_{\infty,\infty} ^{0,L_{m}}}+\underset{m=1,2}{\max} \|\nabla L_{m} ^{-\frac{1}{2}}g\|_{\dot{B}_{\infty,\infty} ^{0,L_{m}}}
\lesssim \|g\|_{\infty}.
$$
If the coefficients $(a_{j})_{j=1,...,2d}$ are $C^{1,1}$, then, 
for all $t\geq 0$ and $m=1,2$,  $\exp(-tL_{m})(1)=1$ in $L^{\infty}$ by Feynman-Kac's formula. 
Therefore \cite[Proposition 6.7]{dy} gives that, for $m=1,2$,
$$
\|L_{m}^{s_{p}}g\|_{BMO_{L_{m}}} \lesssim \|L_{m}^{s_{p}}g\|_{BMO}.
$$
If the coefficients $(a_{j})_{j=1,...,2d}$ are constant, then the assumptions on $g$ reduce to $g \in W^{2s_{p},\infty}$.
In the special case where $L_{1}=L_{2}=-\Delta$, a more general result for pseudo-differential operators has been proven recently in \cite[Theorem 1.1]{R2} for symbols which are $C^{r}$ regular in the spatial variable, with $r>s_{p}$. Even just for multiplication operators, we do not fully recover this result, partly because our abstract setting prevents us from using arguments about the Fourier support of products. In this Section, we are merely demonstrating that adding lower perturbations with smooth enough coefficients is possible.  We intend to develop a more complete perturbation theory in subsequent work.
\end{Remark}

We state our perturbation result for first order perturbations of the wave equation under consideration.

\begin{Cor}
\label{cor:pert}
Let $p \in (1,\infty)$ and $s_p=(d-1)|\frac{1}{p}-\frac{1}{2}|$.  
 Assume that $s_{p} \leq 2$.
For $j=1,...,d$, let $g_{j} \in L^{\infty}$  be such that $g_{j} \in \dot{B}_{\infty,\infty} ^{0,L_{m}}$, 
$\nabla L_{m} ^{-\frac{1}{2}} g_{j} \in \dot{B}_{\infty,\infty} ^{0,L_{m}}$ and $L_{m}^{s_{p}}g_{j} \in BMO_{L_{m}}$
for $m=1,2$. 
Consider 
$$\tilde L: (f,F) \mapsto (L_{1}f,L_{2}F) + 
\sum _{j=1} ^{d} (g_{j}\partial_{j}f,g_{j}\partial_{j}F).$$  For each $t \in \R$, the operator  $(I+\sqrt{\tilde L})^{-s_p} \exp(it\sqrt{\tilde L})$ is bounded on $L^p(\R^d)$. \end{Cor} 

\begin{proof}
 Without loss of generality, we assume that $p \leq 2$ (using duality to get the full result). 
By Theorem \ref{thm:main},  \cite[Example 3.14.15]{ABHN} and Proposition \ref{prop:vertical-conical}, the operator $L$ generates a cosine family on $H^{p}_{FIO,a}(\R^d)$,  with Kisy\'nski space $D(\sqrt{L}) = H^{p,1}_{FIO,a}(\R^d)$ (see \cite{ABHN} for the theory of cosine families).
By Theorem \ref{thm:perturb},  boundedness of Riesz transforms \cite[Corollary 5.19]{AMcT}, and Proposition \ref{prop:vertical-conical}, we have, for all $j=1,...,d$, that
$$\|M_{g_{j}}(\partial_{j} f,\partial_{j}F)\|_{H^{p}_{FIO,a}(\R^d)} \lesssim 
\|(\partial_{j} f,\partial_{j}F)\|_{H^{p}_{FIO,a}(\R^d)}
\lesssim \|(f,F)\|_{H^{p,1}_{FIO,a}(\R^d)} \quad \forall (f,F) \in H^{p,1}_{FIO,a}(\R^d).
$$

We thus obtain from  
\cite[Corollary 3.14.13]{ABHN} that $\exp(it\sqrt{\tilde L}) \in B(H^{p}_{FIO,a}(\R^d))$. 
 Another application of \cite[Corollary 5.19]{AMcT}, also gives that
$$
\|(I+\sqrt{\tilde L})^{-\frac{s_p}{2}} (f,F)\|_{L^p} \sim \|(I+\sqrt{ L})^{-\frac{s_p}{2}} (f,F)\|_{L^p} 
\quad \forall f,F \in W^{1,p},
$$
since $s_{p} \leq 2$. 
Using Lemma \ref{lem:emb2} and Corollary \ref{cor:sob},  we thus have that 
\begin{align*}	
	& 	\|(I+\sqrt{\tilde L})^{-\frac{s_p}{2}} \exp(it\sqrt{\tilde L}) f\|_{L^p}
		\lesssim 
		\|(I+\sqrt{L})^{-\frac{s_p}{2}} \exp(it\sqrt{\tilde L}) f\|_{L^p} \\
		\quad 
		& \lesssim  	\| \exp(it\sqrt{\tilde L}) f\|_{H^p_{FIO,a}(\R^d)} 
		\lesssim \|  f\|_{H^p_{FIO,a}(\R^d)}  \\
		\quad & 
		\lesssim \| (I+\sqrt{L})^{\frac{s_p}{2}}  f\|_{L^p} 
		\lesssim \| (I+\sqrt{\tilde L})^{\frac{s_p}{2}}  f\|_{L^p} \quad \forall f \in L^{p}(\R^{d};\C^{2}).
\end{align*}
\end{proof}

For the proof of Theorem \ref{thm:perturb}, we use the following paraproduct decomposition. 

Let $\Phi \in \calS(\R^d), \phi \in \calS(\R^d)$ with $\phi(0)=1$ and $\Phi_\sigma(\zeta) = \phi(\sigma^{2} |\zeta|^{2})$ for $\sigma>0$, $\zeta \in \R^d$. 
We denote by $M_{\phi(L)g}$ the multiplication operator 
$(f,F) \mapsto (\phi(L_{1})g.f,\phi(L_{2})g.F)$. We denote by $M_{\phi(\underline{L})g}$ the multiplication operator $(f,F) \mapsto (\phi(L_{2})g.f,\phi(L_{1})g.F)$.
\\

For  $f \in \calS_p$  and $g \in \calS(\R^d)$,  we use \eqref{eq:resol-identity} to 
 decompose the product $gf$ as follows. 
\begin{align*}
	M_{g} f 
	& = 
		   \int_1^\infty  M_{\phi(\tau L)g}  \Psi(\tau D_{a})^{2} f \,\frac{d\tau}{\tau} 
		   +  \int_1^\infty  ( M_{g} -M_{ \phi(\tau L)g})  \Psi(\tau D_{a})^{2} f \,\frac{d\tau}{\tau}  \\
		 & \quad  +\int_{S^{d-1}} \int_0^1  M_{\phi(\tau L)g}  \varphi_{\nu}(D_{a})^{2}\Psi(\tau D_{a})^{2} f \,\frac{d\tau}{\tau} d\nu \\
		 & \quad + \int_{S^{d-1}} \int_0^1  ( M_{g} -M_{\phi(\tau L)g}) \varphi_{\nu}(D_{a})^{2}\Psi(\tau D_{a})^{2} f \,\frac{d\tau}{\tau} d\nu.
\end{align*}

Since the two low-frequency terms in the first line are similar but simpler than the two high-frequency terms, we only consider the two latter in the following.  Moreover, note that we can choose $\Phi$ and $\Psi$ such that by integration by parts, the last integral is - up to a low-frequency term - equal to 
\begin{align*}
	 \int_{S^{d-1}} \int_0^1    M_{\Psi_\tau (L) g}   \varphi_{\nu}(D_{a})^{2}\Phi(\tau D_{a}) f \,\frac{d\tau}{\tau} d\nu,
\end{align*}
where $\Psi(\sigma \zeta) =: \psi(\sigma^{2} |\zeta|^{2})$ for $\sigma>0$, $\zeta \in \R^d$.

\begin{Lemma}
\label{lem:para1}
Let $p \in (1,\infty)$.  Let $g \in L^{\infty}$  be such that $g \in \dot{B}_{\infty,\infty} ^{0,L_{m}}$ and 
$\nabla L_{m} ^{-\frac{1}{2}} g \in \dot{B}_{\infty,\infty} ^{0,L_{m}}$ for $m=1,2$. For all $f \in H^p_{FIO,a}(\R^d)$, we have that 
\begin{align*}
	& \| (\omega, \sigma, \cdot) \mapsto \psi_{\omega,\sigma}(D_a) \int_{S^{d-1}} \int_0^1   M_{\phi(\tau L)g}   \varphi_{\nu}(D_{a})^{2}\Psi(\tau D_{a})^{2} f \,\frac{d\tau}{\tau} \,d\nu\|_{L^p(S^{d-1};T^{p,2}(\R^d))} \\
	& \qquad \lesssim  (\|g\|_{\infty}+\underset{m=1,2}{\max} \|g\|_{\dot{B}_{\infty,\infty} ^{0,L_{m}}}+\underset{m=1,2}{\max} \|\nabla L_{m} ^{-\frac{1}{2}}g\|_{\dot{B}_{\infty,\infty} ^{0,L_{m}}}) \|f\|_{H^p_{FIO,a}(\R^d)}. 
\end{align*}
\end{Lemma}

\begin{proof}
We split the integral in $\tau$ into two parts, corresponding to $\tau \in (0,\min(\sigma,1))$ and $\tau \in (\min(\sigma,1),1)$. 
We also split the integral over $S^{d-1}$ into two parts, corresponding to $| \nu\pm\omega| \leq \sqrt{\tau}$ and $| \nu\pm\omega| > \sqrt{\tau}$.
Consider first $\tau \in (0,\min(\sigma,1))$ and $| \nu\pm\omega| \leq \sqrt{\tau}$.
Using Lemma \ref{lem:kernelpsi}, and \cite[Theorem 5.2]{HvNP}, we have that
\begin{align*}
	& \| (\omega, \sigma, \cdot) \mapsto \psi_{\omega,\sigma}(D_a)  \int_0^{\min(1,\sigma)}  
	\int_{| \nu\pm\omega| \leq \sqrt{\tau}} M_{\phi(\tau L)g}  \varphi_{\nu}(D_{a})^{2}\Psi(\tau D_{a})^{2} f d\nu\,\frac{d\tau}{\tau} \|_{L^p(S^{d-1};T^{p,2}(\R^d))} \\
	& \qquad \lesssim \| (\omega, \sigma, \cdot) \mapsto 
	\sigma^{-\frac{d-1}{4}}
	\int_0^{\min(1,\sigma)} 
	\int_{| \nu\pm\omega| \leq \sqrt{\tau}}  M_{\phi(\tau L)g}  \varphi_{\nu}(D_{a})^{2}\Psi(\tau D_{a})^{2} f d\nu \, \frac{d\tau}{\tau} \|_{L^p(S^{d-1};T^{p,2}(\R^d))} \end{align*}
On the other hand, Hardy's inequality implies that 
$$
	(\sigma,\,.\,) \mapsto \int_0^\sigma (\frac{\tau}{\sigma})^{\frac{d-1}{4}}F(\tau,\,.\,)\,\frac{d\tau}{\tau}
$$
is bounded on $T^{p,2}(\R^d)$. We thus have that
\begin{align*}
	& \| (\omega, \sigma, \cdot) \mapsto 
	\sigma^{-\frac{d-1}{4}}
	\int_0^{\min(1,\sigma)} 
	\int_{| \nu\pm\omega| \leq \sqrt{\tau}}  M_{\phi(\tau L)g}  \varphi_{\nu}(D_{a})^{2}\Psi(\tau D_{a})^{2} f d\nu \, \frac{d\tau}{\tau} \|_{L^p(S^{d-1};T^{p,2}(\R^d))} \\
	&  \qquad \lesssim
	\sup_{\tau>0} \|\phi(\tau L) g\|_{\infty} 
 \| (\omega, \tau, \cdot) \mapsto 
	\tau^{-\frac{d-1}{4}}
	\int_{| \nu\pm\omega| \leq \sqrt{\tau}}  \varphi_{\nu}(D_{a})^{2}\Psi(\tau D_{a})^{2} f d\nu \|_{L^p(S^{d-1};T^{p,2}(\R^d))} \\
		&  \qquad \lesssim
	 \|g\|_{\infty} 
 \| (\omega, \tau, \cdot) \mapsto 
	\tau^{-\frac{d-1}{4}}
	\int_{| \nu\pm\omega| \leq \sqrt{\tau}}  \varphi_{\nu}(D_{a})\Psi(\tau D_{a})
	\widetilde{\psi}_{\omega,\tau}(D_{a})
	 f d\nu \|_{L^p(S^{d-1};T^{p,2}(\R^d))},
		\end{align*}

for some $\widetilde{\psi}_{\omega,\tau}$ that satisfies the same assumptions as $\psi_{\omega,\tau}$ in Section \ref{sec:wavepackets}.
Noting that 
$$
\tau^{-\frac{d-1}{4}}
	\int_{| \nu\pm\omega| \leq \sqrt{\tau}} \|\mathcal{F}^{-1}(\psi_{\nu,\tau})\|_{L^{1}}  d\nu \lesssim \tau^{-\frac{d-1}{2}}
	\int_{| \nu\pm\omega| \leq \sqrt{\tau}} d\nu \lesssim 1,
$$
uniformly in $\tau$, we can apply a slight modification of Lemma \ref{lem:kernelpsi}, together with \cite[Theorem 5.2]{HvNP}, and get that
\begin{align*}
& \| (\omega, \tau, \cdot) \mapsto 
	\tau^{-\frac{d-1}{4}}
	\int_{| \nu\pm\omega| \leq \sqrt{\tau}}  \varphi_{\nu}(D_{a})\Psi(\tau D_{a})
	\widetilde{\psi}_{\omega,\tau}(D_{a})
	 f d\nu \|_{L^p(S^{d-1};T^{p,2}(\R^d))},\\
	 	  & \qquad \lesssim 
 \| (\omega, \tau, \cdot) \mapsto 
	\widetilde{\psi}_{\omega,\tau}(D_{a})
	 f  \|_{L^p(S^{d-1};T^{p,2}(\R^d))}
	 	 \lesssim \|f\|_{H^{p}_{FIO,a}(\R^{d})}.
		\end{align*}

We now turn to the part where $\tau \in (0,\min(\sigma,1))$ and $| \nu\pm\omega| > \sqrt{\tau}$. Denoting by $(\omega,\omega_{1},...,\omega_{d-1})$ an orthonormal basis of $\R^{d}$, we remark that, in this region, 
\begin{align*}
\tau(\nu.D_{a}) \psi_{\omega,\sigma}(D_{a}) & = \frac{\tau}{\sigma}
(\nu.\omega) \sigma (\omega.D_{a}) \psi_{\omega,\sigma}(D_{a})
+
\sqrt{\tau} \sqrt{\frac{\tau}{\sigma}} \sum _{j=1} ^{d-1} (\nu.\omega_{j}) \sqrt{\sigma} (\omega_{j}.D_{a}) \psi_{\omega,\sigma}(D_{a})\\
& = \sqrt{\tau}(\frac{\tau}{\sigma}+\sqrt{\frac{\tau}{\sigma}})\widetilde{\psi_{\omega,\sigma}}(D_{a}),
\end{align*}
for some $\widetilde{\psi}_{\omega,\sigma}$ that satisfies the same assumptions as $\psi_{\omega,\sigma}$ in Section \ref{sec:wavepackets} (integrating by parts as in Lemma \ref{lem:kernelpsi}), since
$|\omega.\nu| \lesssim \sqrt{\tau} \leq \sqrt{\sigma}$.
We combine this fact with the following version of the product rule:
$$
M_{\phi(\tau  L)g}(e_{j}.D_{a}) = (e_{j}.D_{a})M_{\phi(\tau \underline{L})g}-M_{(e_{j}.D_{a})\phi(\tau \underline{L})g},
$$
for $j=1,..,d$, where $M_{(e_{j}.D_{a})\phi(\tau \underline{L})g}: (f,F) \mapsto
(-\widetilde{a_{j+d}}\partial_{j}\phi(\tau L_{1})g \cdot F, \widetilde{a_{j}}\partial_{j}\phi(\tau L_{2})g \cdot f)$.

We obtain that, for any $M \in \N$,
\begin{align*}
	& \| (\omega, \sigma, \cdot) \mapsto \psi_{\omega,\sigma}(D_a)  \int_0^{\min(1,\sigma)}  
	\int_{| \nu\pm\omega| > \sqrt{\tau}} M_{\phi(\tau L)g}  \varphi_{\nu}(D_{a})^{2}\Psi(\tau D_{a})^{2} f d\nu\,\frac{d\tau}{\tau} \|_{L^p(S^{d-1};T^{p,2}(\R^d))} \\
	&  \lesssim \underset{j=0,..,2M}{\max}
\| (\omega, \sigma, \cdot) \mapsto \tau^{M}
\widetilde{\psi}_{\omega,\sigma}(D_a) \int \int  M_{(\sqrt{\tau} \nu.D_{a})^{j}\phi(\tau L)g}  \varphi_{\nu}(D_{a})\Psi(\tau D_{a}) \underline{\psi}_{\nu,\tau}(D_a)  f d\nu\,\frac{d\tau}{\tau} \|\\
		&  \qquad + \underset{j=0,..,2M}{\max}
\| (\omega, \sigma, \cdot) \mapsto \tau^{M}
\widetilde{\psi}_{\omega,\sigma}(D_a) \int  \int  M_{(\sqrt{\tau} \nu.D_{a})^{j}\phi(\tau \underline{L})g}  \varphi_{\nu}(D_{a})\Psi(\tau D_{a}) \underline{\psi}_{\nu,\tau}(D_a)  f d\nu\,\frac{d\tau}{\tau} \|,
	\end{align*}
for some $\widetilde{\psi}_{\omega,\sigma}$ and $\underline{\psi}_{\omega,\sigma}$ that satisfy the same assumptions as $\psi_{\omega,\sigma}$ in Section \ref{sec:wavepackets}.
From   Remark \ref{rem:tentbdd}    we know that 
\begin{align*}
	\|(\omega,\sigma,\,.\,) \mapsto   \sigma^{\frac{s_p}{2}} \psi_{\omega,\sigma}(D_a)F(\sigma,\,.\,)\|_{L^p(S^{d-1};T^{p,2}(\R^d))} \lesssim \|F\|_{T^{p,2}(\R^d)}.
\end{align*}
Picking $M>\frac{d-1}{4}+\frac{s_{p}}{2}$, and using Hardy's inequality again, we thus get that  - suppressing a similar estimate with $L$ replaced by $\underline{L}$ - 
\begin{align*}
 &  \underset{j=0,..,2M}{\max}
\| (\omega, \sigma, \cdot) \mapsto \tau^{M}
\widetilde{\psi}_{\omega,\sigma}(D_a) \int \int  M_{(\sqrt{\tau} \nu.D_{a})^{j}\phi(\tau L)g}  \varphi_{\nu}(D_{a})\Psi(\tau D_{a}) \underline{\psi}_{\nu,\tau}(D_a)  f d\nu\,\frac{d\tau}{\tau} \|\\
	& \quad \lesssim 
\underset{j=0,..,2M}{\max}
	\int_{S^{d-1}}\|(\tau,\,.\,) \mapsto \tau^{\frac{d-1}{4}} M_{(\sqrt{\tau} \nu.D_{a})^{j}\phi(\tau L)g} \varphi_{\nu}(D_{a})\Psi(\tau D_{a}) \underline{\psi}_{\nu,\tau}(D_a) f \|_{T^{p,2}(\R^d)} \,d\nu \\
&\quad \lesssim	 (\|g\|_{\infty}+\underset{m=1,2}{\max} \|g\|_{\dot{B}_{\infty,\infty} ^{0,L_{m}}}+\underset{m=1,2}{\max} \|\nabla L_{m} ^{-\frac{1}{2}}g\|_{\dot{B}_{\infty,\infty} ^{0,L_{m}}}) 
	\int_{S^{d-1}}\|(\tau, \cdot)\mapsto  \underline{\psi}_{\nu,\tau}(D_a)f \|_{T^{p,2}(\R^d)} \,d\nu \\
	& \quad  \lesssim  (\|g\|_{\infty}+\underset{m=1,2}{\max} \|g\|_{\dot{B}_{\infty,\infty} ^{0,L_{m}}}+\underset{m=1,2}{\max} \|\nabla L_{m} ^{-\frac{1}{2}}g\|_{\dot{B}_{\infty,\infty} ^{0,L_{m}}})  \|f\|_{H^p_{FIO,a}(\R^d)}.
\end{align*}
For the integral over $\tau \in (\min(\sigma,1),1)$,  we slightly rewrite the above argument, by picking $M \in \N$ such that $M>\frac{d-1}{8}$, and using that $\widetilde{\psi}_{\omega,\sigma}(D_{a}):= \psi_{\omega,\sigma}(D_a) (\sigma^{2}L)^{-M} $ satisfies the same assumptions as $\psi_{\omega,\sigma}$ in Section \ref{sec:wavepackets}. In the region where $| \nu\pm\omega| \leq \sqrt{\tau}$, we first use Lemma \ref{lem:kernelpsi}, \cite[Theorem 5.2]{HvNP}, and Hardy's inequality as before to obtain that 
\begin{align*}
	& \| (\omega, \sigma, \cdot) \mapsto \psi_{\omega,\sigma}(D_a)  \int_{\min(1,\sigma)}  ^{1}
	\int_{| \nu\pm\omega| \leq \sqrt{\tau}} M_{\phi(\tau L)g}  \varphi_{\nu}(D_{a})^{2}\Psi(\tau D_{a})^{2} f d\nu\,\frac{d\tau}{\tau} \|_{L^p(S^{d-1};T^{p,2}(\R^d))} \\
	& \,\lesssim \| (\omega, \sigma, \cdot) \mapsto 
	\sigma^{2M-\frac{d-1}{4}} \sigma^{\frac{d-1}{4}} 
	\widetilde{\psi}_{\omega,\sigma}(D_a)  \int_{\min(1,\sigma)}  ^{1}
	\int_{| \nu\pm\omega| \leq \sqrt{\tau}} L^{M}[M_{\phi(\tau L)g}  \varphi_{\nu}(D_{a})^{2}\Psi(\tau D_{a})^{2} f] d\nu\,\frac{d\tau}{\tau} \| \\	
	& \, \lesssim
		\| (\omega, \tau, \cdot) \mapsto 
	\tau^{2M-\frac{d-1}{4}}
	\int_{| \nu\pm\omega| \leq \sqrt{\tau}}  L^{M}   [ M_{\phi(\tau L)g}\varphi_{\nu}(D_{a})^{2}\Psi(\tau D_{a})^{2} f ]d\nu \|_{L^p(S^{d-1};T^{p,2}(\R^d))} 
 \end{align*}

For $j=1,..,d$, we now use the following version of the product rule:
$$
(e_{j}.D_{a})M_{\phi(\tau L)g} = M_{\phi(\tau  \underline{L})g}(e_{j}.D_{a})+M_{(e_{j}.D_{a})\phi(\tau L)g}.
$$
Let $k \in \{0,...,2M\}$ be even, and $j=1,...,d$.
Letting $\phi_{k}: x \mapsto x^{\frac{k}{2}}\phi(x)$, $m=1,2$, and $\delta \in \{0,1\}$,
we can estimate further by multiples of terms of the form 
\begin{align*}
	&	\| (\omega, \tau, \cdot) \mapsto 
	\tau^{-\frac{d-1}{4}}
	\int_{| \nu\pm\omega| \leq \sqrt{\tau}} M_{\tau^{\delta}(e_{j}.D_{a})^{\delta}\phi_{k}(\tau  L_{m} ) g}   (\tau D_a)^{2M-k} \varphi_{\nu}(D_{a})^{2}\Psi(\tau D_{a})^{2} f d\nu \|_{L^p(S^{d-1};T^{p,2}(\R^d))} \\
& \lesssim 	 
	   \sup_{\tau \in [0,1]} \|(\tau, \cdot) \mapsto  (\tau \partial_{j})^{\delta}(\tau L_{m})^{\frac{k}{2}} \phi(\tau L_{m}) g\|_{L^\infty(\R^d)}\\
& \qquad  \cdot 
 \| (\omega, \tau, \cdot) \mapsto 
	\tau^{-\frac{d-1}{4}} \tau^{\frac{k}{2}}
	\int_{| \nu\pm\omega| \leq \sqrt{\tau}}  \varphi_{\nu}(D_{a})\Psi(\tau D_{a})
	\widetilde{\psi}_{\omega,\tau}(D_{a})
	 f d\nu \|_{L^p(S^{d-1};T^{p,2}(\R^d))},
\end{align*}

for some $\widetilde{\psi}_{\omega,\tau}$ that satisfies the same assumptions as $\psi_{\omega,\tau}$ in Section \ref{sec:wavepackets}. 

For $k \in \{0,...,2M-1\}$ even,   $m=1,2$,  and $j=1,...,d$,  we also obtain multiples of terms of the form
\begin{align*}
	&	\| (\omega, \tau, \cdot) \mapsto 
	\tau^{-\frac{d-1}{4}}
	\int_{| \nu\pm\omega| \leq \sqrt{\tau}} M_{\tau^{\delta}(e_{j}.D_{a})^{\delta}\phi_{k}(\tau  L_{m} ) g}   (\tau D_a)^{2M-k-1} \varphi_{\nu}(D_{a})^{2}\Psi(\tau D_{a})^{2} f d\nu \|_{L^p(S^{d-1};T^{p,2}(\R^d))} \\
& \lesssim 	 
	  \sup_{\tau \in [0,1]} \|(\tau, \cdot) \mapsto  (\tau \partial_{j})^{\delta}(\tau L_{m})^{\frac{k}{2}} \phi(\tau L_{m}) g\|_{L^\infty(\R^d)}\\
& \qquad  \cdot 
 \| (\omega, \tau, \cdot) \mapsto 
	\tau^{-\frac{d-1}{4}} \tau^{\frac{k}{2}}
	\int_{| \nu\pm\omega| \leq \sqrt{\tau}}  (\tau^{2}L)^{M-\frac{k+2}{2}}  (\tau e_{j}.D_{a})^{1-\delta} \tau D_{a}\varphi_{\nu}(D_{a})^{2}\Psi(\tau D_{a})^{2}f
	  d\nu \|_{L^p(S^{d-1};T^{p,2}(\R^d))}.
\end{align*}

The result for the region where $\tau \in (\min(\sigma,1),1)$ and $| \nu\pm\omega| \leq \sqrt{\tau}$
then follows as in the case of the region where $\tau \in (0,\min(\sigma,1))$ and $| \nu\pm\omega| \leq \sqrt{\tau}$.
Finally, we consider the region where $\tau \in (\min(\sigma,1),1)$ and $| \nu\pm\omega| > \sqrt{\tau}$. We first apply the product rule as we did in the region where $\tau \in (0,\min(\sigma,1))$ and $| \nu\pm\omega| > \sqrt{\tau}$ to obtain that,  for any $M' \in \N$, 
\begin{align*}
	& \| (\omega, \sigma, \cdot) \mapsto \psi_{\omega,\sigma}(D_a)  \int_{\min(1,\sigma)}  ^{1}
	\int_{| \nu\pm\omega| > \sqrt{\tau}} M_{\phi(\tau L)g}  \varphi_{\nu}(D_{a})^{2}\Psi(\tau D_{a})^{2} f d\nu\,\frac{d\tau}{\tau} \|_{L^p(S^{d-1};T^{p,2}(\R^d))} \\
	&  \lesssim \underset{j=0,..,2M'}{\max}
\| (\omega, \sigma, \cdot) \mapsto \tau^{\frac{M'}{2}}(\frac{\tau}{\sigma}+\sqrt{\frac{\tau}{\sigma}})^{M'}
\widetilde{\psi}_{\omega,\sigma}(D_a)  \int  M_{(\sqrt{\tau} \nu.D_{a})^{j}\phi(\tau L)g}  \varphi_{\nu}(D_{a})\Psi(\tau D_{a}) \underline{\psi}_{\nu,\tau}(D_a)  f d\nu\,\frac{d\tau}{\tau} \|,	
\\
	&   \quad + \underset{j=0,..,2M'}{\max}
\| (\omega, \sigma, \cdot) \mapsto \tau^{\frac{M'}{2}}(\frac{\tau}{\sigma}+\sqrt{\frac{\tau}{\sigma}})^{M'}
\widetilde{\psi}_{\omega,\sigma}(D_a)  \int  M_{(\sqrt{\tau} \nu.D_{a})^{j}\phi(\tau \underline{L})g}  \varphi_{\nu}(D_{a})\Psi(\tau D_{a}) \underline{\psi}_{\nu,\tau}(D_a)  f d\nu\,\frac{d\tau}{\tau} \|,
	\end{align*}
for some $\widetilde{\psi}_{\omega,\sigma}$ and $\underline{\psi}_{\omega,\sigma}$ that satisfy the same assumptions as $\psi_{\omega,\sigma}$ in Section \ref{sec:wavepackets}. We then fix $M'>s_{p}+\frac{d-1}{2}$, and argue as we did in the region $\tau \in (\min(\sigma,1),1)$ and $| \nu\pm\omega| \leq \sqrt{\tau}$, to obtain that, for all $M>\frac{M'}{2}$, again suppressing  similar terms with $L$ replaced by $\underline{L}$, 
\begin{align*}
&  \underset{j=0,..,2M'}{\max}
\| (\omega, \sigma, \cdot) \mapsto \tau^{\frac{M'}{2}}(\frac{\tau}{\sigma}+\sqrt{\frac{\tau}{\sigma}})^{M'}
\widetilde{\psi}_{\omega,\sigma}(D_a)  \int  M_{(\sqrt{\tau} \nu.D_{a})^{j}\phi(\tau L)g}  \varphi_{\nu}(D_{a})\Psi(\tau D_{a}) \underline{\psi}_{\nu,\tau}(D_a)  f d\nu\,\frac{d\tau}{\tau} \|,	\\
&  \lesssim \underset{j=0,..,2M'}{\max}
\|\tau^{\frac{M'}{2}}(\frac{\tau}{\sigma}+\sqrt{\frac{\tau}{\sigma}})^{M'}\sigma^{2M}
\underline{\widetilde{\psi}}_{\omega,\sigma}(D_a)  \int  L^{M}[M_{(\sqrt{\tau} \nu.D_{a})^{j}\phi(\tau L)g}  \varphi_{\nu}(D_{a})\Psi(\tau D_{a}) \underline{\psi}_{\nu,\tau}(D_a)  f ]d\nu\,\frac{d\tau}{\tau} \|,	\\
	& \lesssim \underset{j=0,..,2M'}{\max}
	\int_{S^{d-1}}\|(\tau,\,.\,) \mapsto \tau^{2M+\frac{M'}{2}-\frac{s_{p}}{2}} L^{M}[M_{(\sqrt{\tau} \nu.D_{a})^{j}\phi(\tau L)g} \varphi_{\nu}(D_{a})\Psi(\tau D_{a}) \underline{\psi}_{\nu,\tau}(D_a) f] \|_{T^{p,2}(\R^d)} \,d\nu \\
	& \lesssim \underset{j=0,..,2M'}{\max}
	\int_{S^{d-1}}\|(\tau,\,.\,) \mapsto \tau^{\frac{d-1}{4}} (\tau^{2}L^{M})[M_{(\sqrt{\tau} \nu.D_{a})^{j}\phi(\tau L)g} \varphi_{\nu}(D_{a})\Psi(\tau D_{a}) \underline{\psi}_{\nu,\tau}(D_a) f] \|_{T^{p,2}(\R^d)} \,d\nu.
\end{align*}
Finally, using the product rule as we did in the region where $\tau \in (\min(\sigma,1),1)$ and $| \nu\pm\omega| \leq \sqrt{\tau}$, we estimate further by terms of the form
\begin{align*}
& \int_{S^{d-1}}\|(\tau,\,.\,) \mapsto \tau^{\frac{d-1}{4}} \varphi_{\nu}(D_{a})\Psi(\tau D_{a}) \underline{\psi}_{\nu,\tau}(D_a) f \|_{T^{p,2}(\R^d)} \,d\nu \\
&\quad \lesssim	
	\int_{S^{d-1}}\|(\tau, \cdot)\mapsto  \underline{\psi}_{\nu,\tau}(D_a)f \|_{T^{p,2}(\R^d)} \,d\nu \lesssim    \|f\|_{H^p_{FIO,a}(\R^d)},
\end{align*}
multiplied by $(\|g\|_{\infty}+\underset{m=1,2}{\max} \|g\|_{\dot{B}_{\infty,\infty} ^{0,L_{m}}}+\underset{m=1,2}{\max} \|\nabla L_{m} ^{-\frac{1}{2}}g\|_{\dot{B}_{\infty,\infty} ^{0,L_{m}}})$.
\end{proof}

For the second paraproduct,  we make use of the following factorisation result for tent spaces (see \cite{cms} for the definition of the tent spaces $T^{p,q}$ when $p = \infty$ or $q \neq 2$).

\begin{Theorem}[{\cite[Theorem 1.1]{CV}}]
\label{thm:tent-factor}
Let $p,q \in (1,\infty)$. If $F \in T^{p,\infty}(\R^d)$ and $G \in T^{\infty,q}(\R^d)$, then $FG \in T^{p,q}(\R^d)$ and 
$$
		\|F \cdot G\|_{T^{p,q}(\R^d)} \leq C \|F\|_{T^{p,\infty}(\R^d)} \|G\|_{T^{\infty,q}(\R^d)},
$$
with a constant $C>0$ which is independent of $F$ and $G$. 
\end{Theorem}

\begin{Lemma}
\label{lem:para2}
Let $p \in (1,\infty)$.  Let $g \in L^{\infty}$ be such that
$L_{m}^{s_{p}}g \in BMO_{L_{m}}$ for $m=1,2$, and let $f \in H^p_{FIO,a}(\R^d)$. Then  
\begin{align*}
	&\|(\omega,\sigma,\cdot)\mapsto\psi_{\omega,\sigma}(D_a)  \int_{S^{d-1}} \int_0^1   M_{\Psi_\tau (L) g}    \cdot \varphi_{\nu}(D_{a})^{2}\Phi(\tau D_{a}) f  \,\frac{d\tau}{\tau} d\nu \|_{L^p(T^{p,2})} \\
& \qquad 	\lesssim   \max_{m=1,2} \|L_{m}^{s_{p}}g\|_{BMO_{L_m}}  \|f\|_{H^p_{FIO,a}(\R^d)}. 
\end{align*}
\end{Lemma}

\begin{proof}
Using Remark \ref{rem:tentbdd} and Hardy's inequality as in the proof of Lemma \ref{lem:para1}, we have that
\begin{align*}
	& \|(\omega,\sigma,\cdot)\mapsto\psi_{\omega,\sigma}(D_a)  \int_{S^{d-1}} \int_0^{\min(\sigma,1)}  M_{\Psi_\tau (L) g}\varphi_{\nu}(D_{a})^{2}\Phi(\tau D_{a}) f \,\frac{d\tau}{\tau} d\nu \|_{L^p(T^{p,2})} \\
	& \quad \lesssim 
	\int_{S^{d-1}}\| (\tau,\cdot)\mapsto\tau^{-\frac{s_p}{2}} M_{\Psi_\tau (L) g}\varphi_{\nu}(D_{a})^{2}\Phi(\tau D_{a}) f \|_{T^{p,2}(\R^d)} \,d\nu.
\end{align*}

Applying Theorem \ref{thm:tent-factor}, the above is bounded by a constant times 
\begin{align*}
	& \|(\tau,\cdot)\mapsto \tau^{-s_p} \Psi_\tau (L) g\|_{T^{\infty,2}(\R^d)}
	\int_{S^{d-1}} \| (\tau,\cdot)\mapsto \tau^{\frac{s_p}{2}}  \varphi_{\nu}(D_{a})^{2}\Phi(\tau D_{a}) f \|_{T^{p,\infty }(\R^d)} \,d\nu \\
	& \quad \lesssim   \max_{m=1,2} \|L_{m}^{s_{p}}g\|_{BMO_{L_m}}\|f\|_{H^p_{FIO,a}(\R^d)},
\end{align*}
where we use \cite[Lemma 4.3]{dy}, and Proposition \ref{prop:nontang} in the last line (together with the fact that $s_{p} \geq \frac{d-1}{2}$).

For the integral over $\tau \in (\min(\sigma,1),1)$, we again have to use the product rule.  With the same arguments as in the proof of Lemma \ref{lem:para1}, we end up with terms of the form 
\begin{align*}
&\|(\tau,\cdot)\mapsto \tau^{-s_p} \Psi_\tau (L) g\|_{T^{\infty,2}(\R^d)}
\\	  & \qquad \qquad \cdot
	\int_{S^{d-1}}\|(\tau,\cdot) \mapsto \tau^{\frac{s_p}{2}}   (\tau^{2}L)^{M-\frac{k}{2}}  (\tau e_{j}.D_{a})^{1-\delta} \varphi_{\nu}(D_{a})^{2}\Phi(\tau D_{a})^{2}f \|_{T^{p,\infty}(\R^d)} \,d\nu \\
& \quad \lesssim   \max_{m=1,2} \|L_{m}^{s_{p}}g\|_{BMO_{L_m}}\|f\|_{H^p_{FIO,a}(\R^d)},
\end{align*}
for $k \in \{0,\ldots,2M\}$ even, and $\delta \in \{0,1\}$ (and similar terms for $k$ odd, as in Lemma \ref{lem:para1}).
\end{proof}

\end{document}